\documentclass{article}
\usepackage{microtype}
\usepackage{graphicx}
\usepackage{subfigure}
\usepackage{booktabs} 

\usepackage{hyperref}

\usepackage[accepted]{icml2021}
\PassOptionsToPackage{numbers, compress}{natbib}

\usepackage{hyperref}
\hypersetup{colorlinks,
            linkcolor=blue,
            citecolor=blue,
            urlcolor=magenta,
            linktocpage,
            plainpages=false}

\usepackage{microtype}
\usepackage{graphicx}
\usepackage{subfigure}
\usepackage{booktabs} 

\usepackage{hyperref}
\usepackage{color}

\usepackage{amsthm}
\usepackage{bbm}
\usepackage{amssymb}
\usepackage{algorithm,algorithmic}
\usepackage{amsmath}               
  {
      \theoremstyle{plain}
      
  }

\DeclareMathOperator*{\argmin}{arg\,min}

\def\reals{{\mathcal R}}

\newcommand{\D}{{\mathcal{D}}}

\newcommand{\K}{\mathcal{K}}
\newcommand{\R}{\mathcal{R}}

\newcommand{\ignore}[1]{}

\def\reals{{\mathbb R}}

\def\bold0{\mathbf{0}}

\newcommand\E{\mbox{\bf E}}

\def\bx{\mathbf{x}}

\def\by{\mathbf{y}}

\newcommand{\B}{\mathbb{B}}

\newcommand{\eps}{\varepsilon}

\newcommand{\kl}[1]{\textbf{\textcolor{cyan}{KL: #1}}}
\newcommand{\iu}[1]{\textbf{\textcolor{blue}{IU: #1}}}

\RequirePackage{filecontents}

\newtheorem{theorem}{Theorem}[section]
\newtheorem{definition}{Definition}[section]
\newtheorem{lemma}{Lemma}[section]
\newtheorem{corollary}{Corollary}[section]
\newcommand{\G}{\mathcal{G}}
\newcommand{\X}{\mathcal{X}}
\newcommand{\Y}{\mathcal{Y}}
\newcommand{\e}{\varepsilon}

\renewcommand{\E}{\mathbb E}
\renewcommand{\R}{\mathbb R}

\newcommand{\barlambda}{\bar{\lambda}}

\renewcommand{\D}{\mathcal{D}}

\newcommand{\teps}{\tilde{\eps}}
\newcommand{\beps}{\bar{\eps}}

\usepackage{cleveref}
\usepackage{hyperref}   

\usepackage{bm}
\usepackage{bbm}
\def\bh{\mathbf{h}}
\renewcommand{\O}{\mathcal{O}}
\renewcommand{\bx}{\bar{x}}
\newcommand{\bg}{\bar{g}}

\newcommand{\Ta}{\mathcal{T}_{\rm{Active}}}
\DeclareMathOperator*{\argmax}{arg\,max}
\usepackage{framed}

\renewcommand{\by}{\bar{y}}

\icmltitlerunning{Fast Projection Onto Convex Smooth Constraints}

\begin{document}

\twocolumn[
\icmltitle{Fast Projection Onto Convex Smooth Constraints
}

\icmlsetsymbol{equal}{*}

\begin{icmlauthorlist}
\icmlauthor{Ilnura Usmanova}{eth-ifa}
\icmlauthor{Maryam Kamgarpour}{ubc}
\icmlauthor{Andreas Krause}{eth-cs}
\icmlauthor{Kfir Yehuda Levy}{vit,technion}
\end{icmlauthorlist}

\icmlaffiliation{eth-ifa}{Automatic Control Laboratory, D-ITET, ETH Zürich,
Switzerland}
\icmlaffiliation{eth-cs}{Department of Computer Science, ETH Zürich, Switzerland}
\icmlaffiliation{technion}{Department of Electrical \& Computer Engineering, Technion - Israel Institute of Technology}

\icmlaffiliation{vit}{A Viterby fellow}

\icmlaffiliation{ubc}{Department of Electrical and Computer Engineering, University of British Columbia, Canada}

\icmlcorrespondingauthor{Ilnura Usmanova}{ilnurau@control.ee.ethz.ch}

\icmlkeywords{Optimization, ICML}
\vskip 0.3in
]
\printAffiliationsAndNotice{}
\begin{abstract}
\looseness -1 The Euclidean projection onto a convex set is an important problem that arises in numerous constrained optimization tasks. Unfortunately, in many cases, computing projections is computationally demanding. In this work, we focus on projection problems where the constraints are smooth and the number of constraints is significantly smaller than the 
dimension.
The runtime of existing approaches to solving such problems is either cubic in the dimension or polynomial in the inverse of the target accuracy.  
Conversely, we propose  a simple and efficient  primal-dual approach, with a runtime that scales only linearly with the dimension, and only logarithmically in the inverse of the target accuracy. 
We empirically demonstrate its performance, and compare it with standard baselines.
\end{abstract}

\section{INTRODUCTION}

Constrained optimization problems arise naturally in numerous fields such as control theory, communication, signal processing, and machine learning (ML). 
A common approach for solving constrained problems is to \emph{project onto the set of constraints} in each step of the optimization method. 
Indeed, in ML the most popular learning method is  \emph{projected} stochastic gradient descent (SGD). Moreover, projections are employed within  projected quasi-Newton \citep{schmidt2009optimizing}, and projected Newton-type methods.

The projection operation in itself requires solving a quadratic optimization problem over the original constraints.
In this work, we address the case where we have several smooth constraints, i.e., our constraint set $\K$ is
\begin{align}
\label{constraintK}
\K =\{x\in\reals^n: h_i(x) \leq 0~;\forall i\in[m]\}~,
\end{align}
where $h_i$'s are convex and smooth. We focus on the case where the dimension of the problem $n$ is \emph{high}, and the number of constraints $m$ is \emph{low}. 
This captures several important ML applications, like multiple kernel learning \citep{ye2007learning}, semi-supervised learning \citep{zhu2006graph}, triangulation in computer vision \citep{aholt2012qcqp}, applications in signal processing \citep{huang2014randomized}, solving constrained MDPs \citep{Altman99constrainedmarkov,jin2020efficiently}.

In some special cases like box constraints, $\ell_2$ or $\ell_1$ constraints, the projection problem can be solved very efficiently. Nevertheless, in general there does not exist a unified and scalable approach for projection. One generic family of approaches for solving convex constrained problems are Interior Point Methods (IPM) \citep{karmarkar1984new,nemirovski2008interior}.  
Unfortunately, in general the runtime of IPMs scales as $O\left(n^3m\log({n}/{\e})\right)$, where $\eps$ is the accuracy of the solution, so these methods  are  unsuitable for high dimensional problems.

\textbf{Our contribution.}
\looseness -1 We propose a generic and scalable approach for projecting onto a small number of  convex smooth constraints.
Our approach applies generally for any constraint set that can be described by Eq.~\eqref{constraintK}. Moreover, our approach extends beyond the projection objective 
to any strongly convex and smooth objective.  
The overall runtime of our method for  finding an approximate projection is $O(n m^{2.5}\log^2(1/\eps) + m^{3.5}\log(1/\eps))$ (see Thm.~\ref{theorem:Main} and the discussion afterwards).
Thus, the runtime of our method scales {\em linearly} with $n$, making it 
highly suitable for solving high-dimensional problems that are ubiquitous in ML. Furthermore, in contrast to the Frank-Wolfe (FW) algorithm \citep{frank1956algorithm}, our approach is generic (i.e., does not require a linear minimization oracle) and depends only {\em logarithmically} on the accuracy. 

\looseness -1 
Moreover, we extend our technique beyond the case of intersections of few smooth constraints.
In particular, we provide a conversion scheme that enables to efficiently project onto norm balls using an oracle that projects onto their dual. One can interpret this result as an algorithmic equivalence between projections onto norm ball and its dual. This holds for both smooth and non-smooth norms.

\looseness -1 On the technical side, our approach utilizes the dual formulation of the problem, and solves it  using a cutting plane method.
Our key observation is that in the special case of projections, one can efficiently compute approximate gradients and values for the dual problem, which we then use within the cutting plane method. Along the way, we prove the convergence of cutting plane methods with approximate gradient and value oracles, which may be of independent interest.

\ignore{
\newpage
\paragraph{Motivation} Projection problem is a very important problem appearing in a range of optimization algorithms such as projected Gradient Descent,  projected quasi-Newton \cite{schmidt2009optimizing}, projected Newton-type methods. It is also used to solve the feasibility problems, and in the Alternating projections algorithm [Boyd et al., 2003] for the determining if the sets intersection is empty or not.

However, there is no unified approach to solve the projection itself. The projection operation can be a nontrivial task, especially in the high dimensional case. In simple cases such as ball or box constraints, the projection problem can be solved analytically.  In more complicated cases, in which the self-concordant barrier of the set exists and known, Interior point methods (IPM) can be applied. For example, IPM is often used for the projection onto intersection of quadratic convex constraints. However, IPM applicability is limited by the relatively small number of variables, since its computational complexity is $O(n^3)$ on the dimensionality of the problem $n$. In the rest cases the choice of the approach to solve projection is dependent on the set and on the solver imagination. Any other method of constrained optimization can be used, for example Sequential Quadratic Programming (SQP), Augmented Lagrangian (AG), etc. For the high dimensional cases with polynomial constraints the first order methods like Conditional Gradient (CD) can be used.

\paragraph{Our contribution}
We propose a unified primal dual approach for the projection operation, suitable for convex sets determined by the smooth convex constraints, and in which the number of constraints is small compared to the number of variables. Moreover, our approach is applicable for the general convex smooth constrained problems with a strongly convex objective and with an access to the first and zeroth information of the objective and constraint functions. For example, a general convex Quadratically Constrained Quadratic Program (QCQP) with the strongly-convex objective is also belongs to this class of problems. 
}

\textbf{Related work.}
In the past years,  projection free first order methods have been extensively investigated.
In particular, the {\em Frank-Wolfe} (FW) algorithm \citep{frank1956algorithm} (also known as conditional gradient method)
is explored, e.g., in  \cite{jaggi2013revisiting,garber2015faster, lacoste2015global,garber2016linear,garber2016faster,lan2016conditional, allen2017linear,lan2017conditional}. 
This approach avoids projections and instead assumes that one can efficiently solve an optimization problem with  \emph{linear} objective over the constraints in each round. Unfortunately, the latter assumption holds only in special cases. In general, this linear minimization might be non-trivial and have the same complexity as the initial problem. Moreover, FW is in general unable to enjoy the  fast convergence  rates that apply to standard gradient based methods \footnote{Note that for some special cases like simplex constraints one can ensure fast rates for FW \citep{garber2013playing,lacoste2015global}. In general, FW requires $O(1/\eps)$ calls to an oracle providing the solution to linear-minimization oracle to obtain $\eps$-accurate solutions.}. In particular, FW does not achieve the linear rate  obtained by  \emph{projected gradient descent} in the case of smooth and strongly-convex problems. Moreover, FW does not enjoy the accelerated rate obtained by \emph{projected Nesterov's method} for smooth and convex problems.


Further popular approaches for solving constrained problems are the Augmented Lagrangian method and ADMM (Alternating Direction Method of Multipliers)
\citep{boyd2011distributed,he20121,goldstein2014fast,eckstein2012augmented}. Such methods work directly on the Lagrangian formulation of the problem while adding penalty terms.  
Under specific conditions, their convergence rate may be linear \citep{nishihara2015general, Giselsson2014MetricSI}. However, 
ADMM requires the ability to efficiently compute the proximal operator, which as a special case  includes the projection operator.
To project onto the intersection of convex constraint sets $\cap_{i=1}^m \mathcal K_i$, consensus ADMM can exploit projection oracles for each $\mathcal K_i$ separately. For this general case, only sublinear rate is shown 
\citep{xu17c, peters2019algorithms}. In the special case of polyhedral sets $\mathcal K_i$, it can have linear rate \citep{Hong2017OnTL}.

\citet{levy19projection} suggest a fast projection scheme that can  approximately solve a projection onto a {\em single} smooth constraint. However, their approach cannot ensure an arbitrarily small accuracy. \citet{li2020fully} extend this approach to a simple constraint like $\ell_1,\ell_{\infty}$-ball in addition to a single smooth constraint.
\citet{basu2017large} address 
 high-dimensional QCQPs (Quadratically Constrained Quadratic Programs)  via a polyhedral approximation of the feasible set obtained by sampling  low-discrepancy sequences. Nevertheless, they only show that their method converges asymptotically, and do not provide any convergence rates. There are also works focusing on fast projections on the sets with a good structure like $\ell_1,\ell_{\infty}$ balls \citep{condat2016fast, gustavo2018fast,li2020fast}.
 

Primal-dual formulation of optimization problems  is a standard tool that has been extensively explored in the literature. 
For example, \citet{arora2005fast}, \citet{Plotkin95fastapproximation} and \citet{lee2015faster} propose to apply the primal-dual approach to solving LPs and SDP. 
Nevertheless, almost all of the previous works consider problems which are either LPs or SDPs, these  are very different from the projection problem that we consider here.   
In particular:~~~~~~\\
\textbf{(i)} These works make use of the specialized structure of LP’s and SDP’s, which does not apply to our work where we consider general constraints. \citet{Plotkin95fastapproximation} consider general convex constraints, but assume the availability of an oracle that can efficiently solve LP’s over this set. This is a very strong assumption that we do not make. 

\textbf{(ii)} We devise and employ an approximate gradient oracle for our dual problem is novel way, which is done by a natural combination of Nesterov’s method in the primal together with a cutting plane method in the dual. 
Furthermore, we provide a novel  analysis for the projection problem, showing  that an approximate solution to the dual  problem can be translated to an approximate primal solution. 

Thus, the techniques and challenges in our paper are very different from the ones in the aforementioned papers. 

\paragraph{Preliminaries and Notation.} We denote the Euclidean norm by $\|\cdot\|$. For a positive integer $t$ we denote $[t] =\{1,\ldots,t\}$. 
A function $F:\reals^n\mapsto \reals$ is \emph{$\alpha$-strongly convex} if, 
$
 F(y)\geq F(x) +\nabla F(x)^\top(y-x)  +\frac{\alpha}{2}\|x-y\|^2,~\forall x,y\in\reals^n~.
$
It is well known that strong-convexity implies $\forall x\in\reals^n$,  
$
\frac{\alpha}{2}\|x-x^*\|^2 \leq F(x)- F(x^*),~\text{where}~ x^*=\argmin_{x\in\reals^n}F(x)~.
$

\section{PROBLEM FORMULATION}\label{sec2}
 The general problem of Euclidean projection is defined as a constrained optimization problem
\begin{align}\label{problem}
  &\min_{x\in \R^n} ~~~~~~~\|x_0  - x\|^2 \nonumber\\
 &\text{subject to }  h_i(x) \leq 0, ~~\forall i =  1,\ldots,m,\tag{P1}
\end{align}
where the constraints $h_i: \R^n \rightarrow \R$ are convex. 

\textbf{Goal:}
Our goal in this work is to find an $\eps$-approximate projection $\bx$, such that for any $x: h_i(x)\leq 0$ we have, 
$
\|\bx  - x_0\|^2 \leq \|x-x_0\|^2 + \eps~,\text{ and }~~
h_i(\bx) \leq \eps~, \forall i\in[m]~.
$

\textbf{Assumptions:}
Defining $\K: = \{x\in \reals^n: h_i(x) \leq 0; ~\forall i\in[m]\}$, we assume that $\K$ is compact. Furthermore, 
we assume the $h_i$'s to be $L$-smooth and $G$-Lipschitz continuous in the convex hull of $\mathcal K$ and $x_0$, i.e.,
$|h_i(x) - h_i(y)|\leq G\|x - y\|~,
~\forall i \in[m]~\forall x,y \in Conv\{\mathcal K, x_0\} \text{ and }
\|\nabla h_i(x) - \nabla h_i(y)\|\leq L\|x - y\|~,
~\forall i \in[m]~\forall x,y \in \R^n.$
We assume both $G$ and $L$ to be known.
  We denote by $H>0$ the bound $
\max_{x\in\K}|h_i(x)| \leq H,~\forall i\in[m]. 
$
We further assume that the distance between $x_0$ and $\K$ is bounded by $B$, 
$
\min_{x\in\K}\|x-x_0\| \leq B.
$ 
Our method does not require the knowledge of $B,H$. 
The Lipschitz continuity and smoothness assumptions above are standard and often hold in machine learning applications.
\paragraph{KKT conditions:}
Our final assumption is that Slater's condition holds, i.e., that there exists a point $x\in\reals^n$ such that $\forall i\in[m];~h_i(x)<0$.  Along with convexity this  immediately
implies that
the optimal solution $x^*$ to Problem~\eqref{problem} satisfies the KKT conditions, i.e., there exist $\lambda_*^{(1)},\ldots \lambda_*^{(m)}\in\reals_+$ s.t. 
$
(x^* -x_0) +\sum_{i=1}^m \lambda_*^{(i)} \nabla h_i(x^*) = 0~, 
\lambda^{(i)}_* h_i(x^*)=0,~\forall i\in[m]~,
$ and
that there exists a finite bound on $|\lambda_*^{(i)} |  ~\forall i\in[m].$
Throughout this paper, we assume the knowledge of an upper bound that we denote by $R$:
$
|\lambda_*^{(i)} | \leq R~;\quad \forall i\in[m]~.
$
In  appendix~\ref{appendix:B}, we showcase two problems where we obtain such a bound explicitly.
When such a bound is unknown in advance, one can apply a generic technique to estimating $R$ ``on the fly", by applying a standard doubling trick. This will only yield a constant factor increase in the overall runtime. We elaborate on this in appendix~\ref{appendix:B2}. 
For simplicity we assume throughout the paper that $R\geq 1$.

\section{FAST PROJECTION APPROACH}\label{sec3}
\vspace{-7pt}
\subsection{Intuition: the Case of a Single Constraint}
\label{sec:SingleConstraint}
\looseness -1 As a warm-up, consider the case of a {\em single} smooth constraint
\begin{align}\label{eq:problem_1}
  &\min_{ x\in \R^n: h(x) \leq 0} ~~~~~~~\|x_0  - x\|^2~.
\end{align}
Our fast projection method relies on the (equivalent) dual formulation of the above problem.
Let us first define the Lagrangian $\forall x\in\reals^n,\lambda\geq0$,
    $
        \mathcal L(x,\lambda): =\|x_0  - x\|^2  + \lambda h(x) ~.
    $
Note that $\mathcal L(\cdot,\cdot)$ is strongly convex in $x$ and concave in $\lambda$. Denoting the dual objective  by $d(\lambda)$, the dual problem is,
 \begin{align}\label{eq:dual}
 \max_{\lambda\geq 0}d(\lambda), ~ \text{where}~ d(\lambda) := \min_{x\in \R^n}\|x_0  - x\|^2  + \lambda h(x).
 \end{align}
We  denote an optimal dual solution by $\lambda_* \in \argmax_{\lambda\geq 0}d(\lambda)$. Our approach is to find an approximate optimal solution to the dual problem $\max_{\lambda\geq0} d(\lambda)$.
Here we show how to do so, and demonstrate how this translates to an approximate solution for the original projection problem (Eq.~\eqref{eq:problem_1}).

\looseness -1 The intuition behind our method is the following.   $\mathcal L(x,\lambda)$ is linear in $\lambda$, and $d(\lambda): = \min_{x \in \R^n}\mathcal L(x,\lambda)$, therefore   $\max_{\lambda\geq0}d(\lambda)$ is a \emph{one-dimensional} concave problem. Moreover, $d(\lambda)$ is differentiable and smooth since the primal problem is strongly convex (see Lemma \ref{lemma:GradientLemmaHighDim}). 
Thus, if we could access an \emph{exact} gradient oracle for $d(\cdot)$, we could use {\em bisection} (see Alg.~\ref{alg:1DimConc} in the appendix) in order to find an $\eps$-approximate solution to the dual problem within $O(\log(1/\eps))$ iterations \citep{JuditskyBook}. Due to strong duality, this translates to an $\eps$-approximate solution of the original problem (Eq.~\eqref{eq:problem_1}). While an exact gradient oracle for $d(\cdot)$ is unavailable, we can efficiently compute {\em approximate} gradients for $d(\cdot)$. Fixing $\lambda\geq 0$, this can be done  by (approximately) solving the following program,
\begin{align}\label{eq:Intuition}
\min_{x\in \reals^n}\mathcal L(x,\lambda):= \|x-x_0\|^2 + \lambda h(x).
\end{align}
Letting $x^*_{\lambda} = \arg\min_{x\in \reals^n} \|x-x_0\|^2 + \lambda h(x)$ one can show that $\nabla d(\lambda) = h(x^*_{\lambda}).$ 
Thus, in order to devise an approximate estimate for $\nabla d(\lambda)$, it is sufficient to solve the
 above \emph{unconstrained} program  in $x$ to within a sufficient accuracy. This can be done at a linear rate using Nesterov's Accelerated Gradient Descent (AGD) (see Alg. \ref{alg:Nes}) due to the fact that  Eq.~\eqref{eq:Intuition} is a smooth and strongly-convex problem (recall that $h(\cdot)$ is smooth). These approximate gradients can then be used instead of the exact gradients of $d(\cdot)$ to find an $\eps$-optimal solution to the dual problem within $O(\log(1/\eps))$ iterations. The formal description for the case of a single constraint can be found in Appendix \ref{appendix:A}. Next we discuss   our approach for the case with several constraints.

\paragraph{Remark} Note that using Nesterov's AGD method for the dual problem in the same way as we do for the primal problem sounds like a very natural idea. However, we cannot guarantee the strong concavity of the dual problem and hence, we cannot hope for the linear convergence rate of this approach. In contrast, the bisection algorithm can guarantee the linear convergence rate even for non-strongly concave dual problems. 

\subsection{Duality of Projections Onto Norm Balls}
\label{sec:DualProject}
As a first application, we show how the approach from Section~\ref{sec:SingleConstraint} has applications for efficient projection on norm balls.
The dual of a norm is an important notion that is often used in the analysis of algorithms. Here we show  an  algorithmic connection between the projection onto norms and onto their dual.
Concretely, we show that one can use our framework in order to obtain \emph{an efficient conversion scheme} that enables to project onto a given unit norm ball  using an oracle that enables to project onto its dual norm ball.
This applies even if the norms are non-smooth, thus extending our technique beyond constraint sets that can be expressed as an intersection of few smooth constraints.
Our approach can also be generalized to general convex sets and their dual (polar) sets.

Given a  norm $P:\reals^n\mapsto \reals$, its dual norm is defined as,
$$
P_*(x): = \max_{P(z)\leq 1} z^\top x~;~~\forall x\in\reals^n
$$
As an example, for any $p\geq 1$ the dual of the $\ell_p$-norm is the  $\ell_q$-norm with $q=p/(p-1)$. Furthermore, the dual of the spectral norm (over matrices) is the nuclear norm; finally  for a PD matrix $A\in \reals^{d\times d}$ we can define the induced norm $\|x\|_A=x^\top A x$, whose dual is $(\|x\|_A)_*:=\|x\|_{A^{-1}}:=x^\top A^{-1} x$. 

Our goal is to project onto the norm ball w.r.t. $P(\cdot)$, i.e.,
\begin{align}\label{eq:problem_NormBall}
  &\min_{ x\in \R^n:P(x) \leq 1} ~~~~~~~\|x_0  - x\|^2~.
\end{align}
Next we state our main theorem for this section,
\begin{theorem}\label{thm:DualProj}
Let $P(\cdot)$ be a norm, and assume that we have an oracle that enables to project onto its dual norm ball $P_*(\cdot)$.
Then we can find an $\eps$-approximate solution to Problem~\eqref{eq:problem_NormBall}, by using $O(\log(1/\eps))$ calls to that oracle.
\end{theorem}
The idea behind this conversion scheme between norm ball projections  is to start with the dual formulation of the problem as we describe in Eq.~\eqref{eq:dual}. Interestingly, one can show that the projection oracle onto the dual norm, enables to compute the \emph{exact} gradients of $d(\lambda)$ in this case. This in turn enables to find an approximate solution to the dual problem using only logarithmically many calls to the dual  projection oracle. Then we can show that such a solution can be translated to an approximate primal solution.
We elaborate on our approach in Appendix~\ref{appendix:DualProj}.

\vspace{-7pt}
\subsection{Projecting onto the Intersection of  Several Non-Linear Smooth Constraints.}
\label{sec:HighDimcase}
\looseness -1 In the rest of this section we will show how to extend our method from  Section~\ref{sec:SingleConstraint} to problems with several constraints.
Similarly to Section~\ref{sec:SingleConstraint}, we solve the dual objective using approximate gradients, which we obtain by running Nesterov's method over the primal variable $x$.
Differently from the one-dimensional case, the dual problem is now {\em multi-dimensional}, so we cannot use bisection. Instead, we employ {\em cutting plane methods} like center of gravity \citep{levin1965algorithm,newman1965location}, the Ellipsoid method \citep{shor1977cut,iudin1977evaluation}, and Vaidya's method \citep{vaidya1989new}.
These methods are especially attractive in our context, since their convergence rate depends only  {\em logarithmically} on the accuracy, and their runtime is {\em linear} in the dimension $n$.
Our main result, Theorem~\ref{theorem:Main}, states that we find an $\eps$-approximate solution to the projection problem \eqref{problem} within a total runtime of  $O\left( n m^{3.5} \log(m/\eps) +m^4 \log(m/\eps)\right)$ if we use the classical Ellipsoid method, and a runtime of $O\left( n m^{2.5} \log(m/\eps) +m^{3.5}\log(m/\eps) \right)$ if we use the more sophisticated  method by \citet{vaidya1989new}. 

The Lagrangian of the original problem~\eqref{problem} is defined as follows:  $\forall x\in\reals^n, \lambda^{(1)},\ldots,\lambda^{(m)}\geq0 $,
$
\mathcal L(x,\lambda) := \|x-x_0\|^2 + \lambda^\top \bh(x)~,
$
where $\lambda: = (\lambda^{(1)},\ldots, \lambda^{(m)}), \bh(x): = (h_1(x),\ldots,h_m(x))\in \reals^m$.
Defining $d(\lambda): = \min_{x\in\reals^n}\mathcal L(x,\lambda)$, the dual problem is now defined as follows,
\begin{align}\label{dual111}
&\max_{\lambda\in\reals^m, \lambda\geq  0} d(\lambda)~,
\end{align} 
where  $\lambda\geq 0$ is an  elementwise inequality.
Recall that we  assume that we are given  $R\geq0$ such that $\lambda_*\in \{\lambda: \|\lambda\|_\infty \leq R\}$, for some $ \lambda_*\in \argmax_{\lambda\geq 0}d(\lambda)$.
 Thus, our dual problem can be written as 
\begin{align}\label{dual}
&\max_{\lambda \in\D} d(\lambda) \tag{P2}~,
\end{align} 
where $\D: = \{\lambda\in\reals^m:~\forall i\in[m];~ \lambda^{(i)}
\in [  0,R]  \}$, and $\lambda^{(i)}
$ is the $i^{\rm{th}}$ component of $\lambda$.
Thus, $\D$ is an  $\ell_{\infty}$-ball of diameter $R$ centered at $[R/2,\ldots,R/2]^T$. In the rest of this section, we describe and analyze the two components of our fast projection algorithm. In Sec.~\ref{sec:CuttingPlane} we describe the first component, which is a cutting plane method that we use to solve the dual objective. In contrast to the standard cutting plane approach where {\em exact} gradient and value oracles are available, we describe and analyze a setting with {\em approximate oracles}.  Next, in Sec.~\ref{sec:ApproximateOracleHD} we show how to construct approximate gradient and 
value oracles using a fast first order method.
Finally, in Sec.~\ref{sec:HDMain} we show how to combine these components to our fast projection algorithm that approximately solves the projection problem. 

\subsection{Cutting Plane Scheme with Approximate Oracles}
\label{sec:CuttingPlane}
We first describe a general recipe for cutting plane methods, which captures the center of gravity, Ellipsoid method, and Vaidya's method amongst others. Such schemes require access to exact gradient and value oracles for the objective.
Unfortunately, in our setting we are only able to devise {\em approximate} oracles. To address this issue, we provide a generic analysis, showing that every cutting plane method converges {\em even when supplied with approximate oracles}. This result may be of independent interest, since it will enable the use of Cutting plane schemes in ML application where we often only have access to approximate oracles.

\textbf{Cutting Plane Scheme:}
\looseness -1 We seek to solve $\max_{\lambda\in\D}d(\lambda)$, where $\D$ is a compact convex set in $\reals^m$ and $d(\cdot)$ is concave. 
Now, assume that we may access an exact separation oracle $\O_s$ for $\D$, that is,  for any $\lambda_t \notin \D$, $\O_s$ outputs  $w\in\reals^m$ such that
$\D \subseteq \{\lambda \in \R^m: w^\top(\lambda-\lambda_t)\leq 0 \}$. We also assume access to 
$(\eps_g,\eps_v)$-approximate
 gradient and value oracles $\O_g:\D\mapsto \reals^m,\O_v:\D\mapsto \reals$ for $d(\cdot)$, meaning,
$
\| \nabla d(\lambda) - \O_g(\lambda)\|\leq \eps_g, ~~ 
|  d(\lambda) - \O_v(\lambda)| \leq \eps_v.
$ 
Finally, assume that we are given a point 
$\lambda_1 = [R/2,\ldots,R/2]\in \D $, and $R>0$ such that $\D\subseteq M_1:=\{\lambda\in\R^m_+: \|\lambda-\lambda_1\|_\infty \leq R/2\}$.  A cutting plane method works as demonstrated
in Alg.~\ref{alg:CuttingPlane}.
\begin{algorithm}[ht]
\caption{Cutting Plane Method with Approximate Oracles}
\label{alg:CuttingPlane}
\begin{algorithmic}
\STATE \textbf{Input}: gradient and value oracles $\O_g,\O_v$ with accuracies $(\eps_g,\eps_v)$, and exact separation oracle $\O_s$ 
\FOR{ $t\in[T]$} 
\IF{$\lambda_t\in \D$}\STATE call gradient oracle $g_t \gets \O_g( \lambda_t)$, set $w_t = -g_t$; 
\ELSE \STATE call separation oracle and set $w_t \gets \O_s( \lambda_t)$. 
\ENDIF
\STATE  Construct $M_{t+1}$ such that $\{\lambda \in M_t: w_t^\top(\lambda-\lambda_t) \leq 0\}\subseteq M_{t+1}$, and choose $\lambda_{t+1}\in M_{t+1}$.
\ENDFOR
\STATE \textbf{Output}: $\barlambda \in \argmax_{\lambda\in \{\lambda_1,\ldots,\lambda_T\}\cap\D}\O_v(\lambda).$
\end{algorithmic}
\end{algorithm}

\textbf{Remark 1:} The output of the scheme in Alg.~\ref{alg:CuttingPlane} is always non-empty since we have assumed $\lambda_1\in\D$. 

Cutting plane methods differ from each other by the construction of sets $M_t$'s and choices of query points $\lambda_t$'s. 
For such methods, the volume of $M_t$'s decreases {\em exponentially} fast with $t$, and this gives rise to linear convergence guarantees in the case of exact gradient and value oracles. 
\begin{definition}[$\theta$-rate Cutting Plane method]
\label{def:rateCP}
We say that a cutting plane method has rate $\theta>0$ if the following applies:
$
\forall t,~\text{Vol}(M_t)/\text{Vol}(M_1)\leq e^{-\theta t}~;
$
and $\text{Vol}$ is the usual $m$-dimensional volume.
\end{definition}
For example, for the {\em center of mass method} as well as {\em Vaidya's method}, we have $\theta = O(1)$, for the {\em Ellipsoid method} we have $\theta = O(1/m)$. Our next lemma  extends the convergence of cutting plane methods to the case of approximate oracles. Let us first denote $\D_\eps: = \{ \lambda\in\D: d(\lambda)\geq d(\lambda_*)-\eps\}$ the set of all $\e$-approximate solutions
, where $\lambda_*\in\argmax_{\lambda\in\D}d(\lambda)$. We need $\D_\eps$ to have nonzero volume to ensure the required accuracy after the sufficient decrease of volume of $M_t$. 
Later we 
show that in our case with Lipschitz continuous and  convex $h_1,\ldots,h_m$, then $\D_\eps$  
contains $\ell_{\infty}$-ball of radius $r(\eps) \propto \eps/m$  (Corollary~\ref{corollary:SetSolutions}).
\begin{lemma}\label{lemma:CuttingApproximate}
Let $\lambda_1\in\D, R>0$ such, $\D\subseteq \{\lambda: \|\lambda-\lambda_1\|_\infty \leq R/2\}$.
Given $\eps>0$ assume that there exists an $\ell_{\infty}$-ball of diameter $r(\eps)>0$ that is contained in $\D_\eps.$ 
Now assume that $d(\lambda)$ is concave and we use the cutting plane scheme of Alg.~\ref{alg:CuttingPlane}  with oracles that satisfy $\eps_g\leq \frac{\eps}{R\sqrt{m}}$, and $\eps_v \leq \eps$ . Then after $T = O(\frac{m}{\theta}\log(R/r(\eps)))$ rounds it outputs $\barlambda\in\D$ such that,
$
 \max_{\lambda\in\D} d(\lambda) - d(\barlambda)  \leq 4\eps,
$
where $\theta$ is the rate of the cutting plane method. 
\end{lemma}
\begin{proof}
We denote  $\Ta = \{t\in[T]: \lambda_t\in \D\}$, clearly this set is non-empty since $\lambda_1\in\D$.
Also, for any $t\in\Ta$ we denote $g_t: = \O_g(\lambda_t)$ (note that in this case $w_t =-g_t$). We divide the proof into two cases: when $\D_{\e}$ is separated by $w_t$ from all $\lambda_t\in\D$, and when not.

\textbf{Case 1:} Assume that there exists $t\in \Ta$, and $\lambda_\eps \in\D_\eps$ such that, $w_t^\top (\lambda_\eps-\lambda_t)= g_t^\top (\lambda_t-\lambda_\eps)\geq 0 $.
In this case, using the concavity of $d(\cdot)$ and definitions of $g_t, R$, we get,
$
d(\lambda_t) 
\geq d(\lambda_\eps) +\nabla d(\lambda_t)^\top(\lambda_t-\lambda_\eps) 
= 
d(\lambda_\eps) +g_t^\top(\lambda_t-\lambda_\eps) + (\nabla d(\lambda_t)-g_t)^\top(\lambda_t-\lambda_\eps) \geq
d(\lambda_*) -\eps -R\sqrt{m}(\eps/(R\sqrt{m})) = d(\lambda_*) -2\eps,$  
where we used $\|y\|_2\leq \sqrt{m} \|y\|_\infty,~\forall y\in\reals^m$.
Thus, $d(\barlambda)\geq \O_v(\barlambda)-\eps \geq \O_v(\lambda_t) -\eps\geq d(\lambda_t) - 2\eps \geq d(\lambda_*) -4\eps$. 

\textbf{Case 2:} Assume that for any $t\in \Ta$, and any $\lambda_\eps\in\D_\eps$, we have $w_t^\top (\lambda_\eps-\lambda_t)= g_t^\top (\lambda_t-\lambda_\eps)\leq 0$.
This 
implies that $\forall t\in[T], \forall \lambda_\eps \in\D_\eps$, 
$w_t^\top (\lambda_\eps-\lambda_t)\leq 0$. Hence 
$\forall t\in[T],~\D_\eps\subseteq M_t$, implying that 
\vspace{-0.cm}
 \begin{align}\label{eq:VolRelation}
 \forall t\in[T]~\text{Vol}(\D_\eps) \leq \ \text{Vol}(M_t).
\end{align}
\looseness -1 Next, we show that the above condition can hold only  if $T\leq \frac{m}{\theta}\log(R/r(\eps))$. 
Indeed, according to our assumption $\text{Vol}(\D_\eps) \geq \text{Vol}(\ell_{\infty}\text{-ball of radius} ~r(\eps) ) = r^{m}(\eps)$. 
On the other hand, we assume that $\text{Vol}(M_t)\leq e^{-\theta t} \text{Vol}(\ell_{\infty}\text{-ball of radius} ~R/2 ) = e^{-\theta t} (R/2)^m$. Combining these with Eq.~\eqref{eq:VolRelation} implies that in order to satisfy \textbf{Case 2}, we must have $T\leq \frac{m}{\theta}\log(R/2r(\eps))$. Thus, for any $T>\frac{m}{\theta}\log(R/r(\eps))$, \textbf{Case 1} must hold, which establishes the lemma.
\end{proof}
\vspace{-10pt}
\subsection{Gradient and Value Oracles for the Dual}
\label{sec:ApproximateOracleHD} 
Here we show how to efficiently devise gradient and value oracles for the dual objective.
Our scheme is described in Alg.~\ref{alg:MinimizingPrimalHighDim}.
Similarly to  the one-dimensional case, given $\lambda$ we approximately minimize $\mathcal \mathcal L(\cdot,\lambda)$, which enables us to derive approximate gradient and value oracles.
The guarantees of Alg.~\ref{alg:MinimizingPrimalHighDim} are given in Lemma~\ref{lemma:GradOracleLemmaHD}. Before we state the guarantees of Alg.~\ref{alg:MinimizingPrimalHighDim} we derive a closed form formula for $\nabla d(\lambda)$. 
Recall that,  $d(\lambda) =\min_{x\in\reals^n}  \mathcal L(x,\lambda)$, and that  $\mathcal L(x,\lambda)$ is $2$-strongly-convex in $x$.
This implies that the minimizer of $\mathcal L(\cdot,\lambda)$ is unique, and we therefore denote,
\begin{align} \label{eq:ExactGradDualHD}
x^*_\lambda: = \argmin_{x\in \reals^n} \mathcal L(x,\lambda)~.
\end{align}
The next lemma shows we can compute $\nabla d(\lambda)$ based on $x^*_\lambda$, and states the smoothness of $\nabla d(\lambda)$.

\begin{algorithm}[H]
\caption{$\O$- approximate gradient/value oracles for $d(\cdot)$ }
\label{alg:MinimizingPrimalHighDim}
\begin{algorithmic}
\STATE \textbf{Input}:  $\lambda\geq 0$,  target accuracy $\teps$
\STATE Compute $x_{\lambda}$, an $\teps$-optimal solution of
$
\min_{x\in\reals^n} \mathcal L(x,\lambda): = \|x-x_0\|^2 + \lambda^\top \bh(x)~.
$
\STATE \textbf{Method:} 
Nesterov's AGD (Alg.~\ref{alg:Nes}) with $\alpha=2, \beta = 2+\|\lambda\|_1 L$, and 
$T = O(\sqrt{\beta}\log(\beta B/\tilde\eps))$~.
\STATE \textbf{Let}:~~$v:= \|x_{\lambda}-x_0\|^2 + \lambda^\top \bh(x_{\lambda})$,
$g: = \bh(x_{\lambda})$
\STATE \textbf{Output}: $(x_{\lambda}, g,v)$
\end{algorithmic}
\end{algorithm}

\begin{algorithm}[H]
\caption{Accelerated Gradient Descent (AGD) 
\cite{nesterov1998introductory}
}
\label{alg:Nes}
\begin{algorithmic}
\STATE \textbf{Input}: 
$F:\reals^n \rightarrow \reals$,  $x_0\in\reals^n$, iterations 
$T$, strong-convexity $\alpha$, smoothness  $\beta$
\STATE {Set}: $y_0 = x_0$, $\kappa: = \beta/\alpha$  
\FOR{$t=0,\ldots, T-1$} 
\STATE 
$y_{t+1} = x_t-\frac{1}{\beta} \nabla F(x_t)~,$\\
$x_{t+1} = \left(1+\frac{\sqrt{\kappa}-1}{\sqrt{\kappa}+1} \right)y_{t+1} -  \frac{\sqrt{\kappa}-1}{\sqrt{\kappa}+1} y_t~.$
\ENDFOR
\STATE \textbf{Output}: $y_T$
\end{algorithmic}
\end{algorithm}
\begin{lemma}\label{lemma:GradientLemmaHighDim}
For any $\lambda\geq 0$ the following holds:\\ \textbf{(i)}  $\nabla d(\lambda) = \bh(x^*_\lambda)$, and $\forall \lambda_1,\lambda_2\geq0$,
$$
\|\nabla d(\lambda_1)-\nabla d(\lambda_2)\|\leq  mG^2 \|\lambda_1-\lambda_2\|~;
$$
\vspace{-0.55cm}
$$
~~\text{and},~~
\|x^*_{\lambda_1} -   x^*_{\lambda_2}\|\leq \sqrt{m}G\|\lambda_1-\lambda_2\|.
$$
Also, \textbf{(ii)} $
d(\lambda_*) - d(\lambda) \leq m^2G^2\|\lambda -\lambda_*\|_\infty^2 + mH\|\lambda-\lambda_*\|_\infty.
$
 Moreover, \textbf{(iii)}  $x^* = x^*_{\lambda_*}$, where  $\lambda_*,x^*$ are the optimal solutions to the dual and primal problems.
\end{lemma}

The proof is quite technical and can be found in Appendix \ref{Proof_lemma:GradientLemmaHighDim}. 
From the above lemma we can show that for any $\eps$ there exists an $\ell_{\infty}$-ball of a sufficiently large radius $r(\eps)$ contained in the set of $\eps$-optimal solutions to the dual problem in $\D$. 
\begin{corollary} \label{corollary:SetSolutions}
Let $\eps\in [0, 1]$. Then there exists an $\ell_{\infty}$-ball of radius $r(\eps): = (2m)^{-1}\min\{\eps/H,\sqrt{\eps}/G\}$ that is contained in the set of $\eps$-optimal solutions within $\D$.
\end{corollary}
The proof is in Appendix \ref{proof:corollary1}. 
Eq.~\eqref{eq:ExactGradDualHD} together with Lemma~\ref{lemma:GradientLemmaHighDim} suggest that exactly minimizing $\mathcal L(\cdot,\lambda)$ enables to obtain gradient and value oracles for $d(\cdot)$. In Alg~\ref{alg:MinimizingPrimalHighDim} we do so approximately, and the next lemma shows that this translates to approximate oracles.

\begin{lemma}\label{lemma:GradOracleLemmaHD}
Given, $\lambda\geq0$, running Alg.~\ref{alg:MinimizingPrimalHighDim} it outputs, $(x,g,v)$ such that the following applies: 
\begin{align*}
&\textbf{(i)}~~\|g-\nabla d(\lambda)\| \leq \sqrt{mG^2\teps}~; ~
 ~
\textbf{(ii)}~~\|x-x^*_\lambda\|^2\leq \teps~; ~ \\
&~~\text{ and }~
\textbf{(iii)}~~|v-d(\lambda)|\leq \teps~.
\end{align*}
 Additionally,  Alg.~\ref{alg:MinimizingPrimalHighDim} requires $T_{\rm{Internal}} = O(\sqrt{1+mR L}\log(m/\teps))$ queries for the gradient of $\bh(\cdot)$, and its  total runtime is $O(nmT_{\rm{Internal}})\approx O(nm^{3/2}\log(m/\teps))$.
\end{lemma}
The proof is in Appendix \ref{proof:lemma4.3}.  The proof of the first part is based on $2$-strong-convexity of $\mathcal L(\cdot,\lambda)$ and $G$-Lipschitz continuity of $\textbf{h}(\cdot)$. 
The second part of the above result also uses the convergence rate of Nesterov's AGD \cite{nesterov1998introductory} described in Appendix in Theorem \ref{thm:nesterov}. 
Using the notation that appears in the description of the cutting plane method (Alg.~\ref{alg:CuttingPlane}) we can think of Alg.~\ref{alg:MinimizingPrimalHighDim} as a procedure that receives $\lambda\geq 0$ and returns 
a gradient oracle $\O_g(\lambda):=g$, value oracle $\O_v(\lambda):=v$, and primal solution oracle $\O_x(\lambda):=x_{\lambda}$.

\textbf{Remark:}
Notice that scaling the constraints $\bh$ by a factor $\alpha>0$ leaves the constraints set unchanged, while scaling the smoothness $L$  by a factor of $\alpha$.
Nonetheless, this naturally also scales the bound of the Lagrange multipliers, $R$, by a factor of $1/\alpha$. 
Lemma~\ref{lemma:GradOracleLemmaHD} tells us that the runtime of our algorithm, $T_{\rm{Internal}}$, depends only on $RL$ and is therefore {\em invariant} to such scaling. 

\subsection{Fast Projection Algorithm}
\label{sec:HDMain}
Below we describe how to compose the two components presented in Sections~\ref{sec:CuttingPlane} and~\ref{sec:ApproximateOracleHD} to a complete algorithm for solving the projection problem of \eqref{problem}.
\begin{algorithm}[H]
\caption{\textbf{Fast Projection Method} }
\label{alg:Main}
\begin{algorithmic}
\STATE \textbf{Input}: Accuracy parameters $\teps>0$, $\lambda_1\in D$, number of rounds $T$
\STATE \textbf{(1)} For any $\lambda \in \R^m$  define three oracles: 
$
\O_g(\lambda):= g,~\O_v(\lambda):= v,\text{ and }~\O_x(\lambda):= x_{\lambda}
$ according to the output $(g,v,x_{\lambda})$ of  Alg.~\ref{alg:MinimizingPrimalHighDim} with the inputs $\lambda$ and $\teps$, 
\STATE \textbf{(2)} Define the separation oracle $
\O_s(\lambda)^i:= [1,\text{if } \lambda^{(i)}>R; 0,\text{if }\lambda^{(i)}\in (0,R); -1, \text{if } \lambda^{(i)}<0]$,
\STATE \textbf{(3)} Employ a cutting plane method as in  Alg.~\ref{alg:CuttingPlane} for solving the dual problem, $\max_{\lambda\in\D}d(\lambda)$,
\STATE \textbf{Output:}  $\barlambda \in \argmax_{\lambda\in \{\lambda_1,\ldots,\lambda_T\}\cap\D}\O_v(\lambda_t)$, and $\bx = \O_x(\barlambda)$.
\end{algorithmic}
\end{algorithm}
Our method in Alg.~\ref{alg:Main} employs a cutting plane scheme (Alg.~\ref{alg:CuttingPlane}), while using Alg.~\ref{alg:MinimizingPrimalHighDim} in order to devise the gradient and value oracles for $d(\cdot)$. 
Next we discuss the role of   the primal solution oracle $\O_x$, and connect it to our overall projection scheme.
Recall that the cutting plane method that we use above finds $\barlambda$, which is an approximate solution to the dual problem.
To extract a primal solution from the dual solution $\barlambda$,  it makes sense to approximately solve $\min_{x\in\reals^n}\mathcal L(x,\barlambda)$, and this is exactly what 
the oracle $\O_x$ provides (see Alg.~\ref{alg:MinimizingPrimalHighDim}). 
Next we state the guarantees of the above scheme.
\begin{theorem}
\label{theorem:Main}
Let $\eps>0$, and consider the projection problem of \eqref{problem}, and its dual formulation in Eq.~\eqref{dual111}. Then upon invoking the scheme in Alg.~\ref{alg:Main} with $\teps = \frac{\eps^4}{256(mRG)^6}$, and $T = O(\frac{m}{\theta}\log(mR/\eps))$, it outputs $\bx$ such that $\forall x\in\K: = \{x: \bh(x)\leq 0\}$, 
$$
\|\bx-x_0\|^2 \leq \|x-x_0\|^2 + 6\eps;~~\text{ and }~~
h_i(\bx) \leq \eps,~~ \forall i\in[m]~.
$$ 
Moreover, the total runtime of our method is $O\left(n m^{2.5}{\theta^{-1}} \log^2(m/\eps) + \tau_{\rm{CP}}(m)m\theta^{-1}\log(mR/\eps)\right)$, where $\theta$ is the rate of the cutting plane method (Def.~\ref{def:rateCP}), and $\tau_{\rm{CP}}(m)$ is the extra runtime required by the cutting plane method for updating the sets $M_t$ beyond calling the gradient and value oracles.
\end{theorem}

Let us discuss two  choices of a  cutting plane method:\\
\textbf{Ellipsoid method:} In this case $\theta = O(1/m)$ and $\tau_{\rm{CP}}(m) = O(m^2)$. Thus, when used within our scheme the total runtime is $O\left( n m^{3.5} \log(m/\eps) +m^4 \log(m/\eps)\right)$.\\
\textbf{Vaidya's method:} In this case $\theta = O(1)$ and $\tau_{\rm{CP}}(m) = O(m^{2.5})$. Thus, when used within our scheme the total runtime is $O\left( n m^{2.5} \log(m/\eps) +m^{3.5}\log(m/\eps) \right)$.
\begin{proof}[Proof of Thm.~\ref{theorem:Main}]
First notice that we may apply the cutting plane method of Alg.~\ref{alg:CuttingPlane} since $\D$ is an $\ell_{\infty}$-ball of diameter $R$, so we can set $M_1: = \D$, and $\lambda_1$ as its center.
Moreover, according to Corollary~\ref{corollary:SetSolutions} for any $\eps\geq 0$ there exists $r \propto \eps/m$ such that an $\ell_{\infty}$-ball of radius $r$ is contained in the set of $\eps$-optimal solutions to the dual problem in $\D$. Let us denote $\beps: = \left(\frac{\eps}{{4}mRG}\right)^2$, and notice that we can write
$\teps = \left(\frac{\beps}{mRG}\right)^2$.
Now by setting $\teps $ as accuracy parameter to Alg.~\ref{alg:MinimizingPrimalHighDim}, it follows from Lemma~\ref{lemma:GradOracleLemmaHD} that it generates gradient and value oracles with the following accuracies,
$
\eps_g = \sqrt{mG^2\teps} \leq \frac{\beps}{R\sqrt{m}}~;~~
\text{ and }
~\eps_v \leq \teps\leq \beps~.
$ 
Now applying Lemma~\ref{lemma:CuttingApproximate} with these accuracies implies that within $T = \frac{m}{\theta}\log(mR/\eps)$ calls to these approximate oracles it outputs a solution $\barlambda$ such that $d(\barlambda)\geq d(\lambda_*)-4\beps$.
Next we show that this guarantee on the dual translates to a guarantee for $\bx$ w.r.t. the original primal problem \eqref{problem}.  We will require the following lemma, proved in Appendix \ref{proof:lemma4.4}.
\begin{lemma} \label{lemma:GradIneqSmooth}
Let $F:\reals^m \rightarrow \reals$ be an $L$-smooth and concave function, and let $\lambda_* =\argmax_{\lambda\in \D }F(\lambda)$. Also let $\D$ is a convex subset of ~$\reals^m$. Then,
$ \| \nabla F(\lambda) -\nabla F(\lambda_*)\|^2 \le 2L \left( F(\lambda_*) - F(\lambda)\right), \quad \forall \lambda\in \D~.$
\end{lemma} 
Using the above lemma together with the $mG^2$-smoothness of $d(\cdot)$ (Lemma~\ref{lemma:GradientLemmaHighDim}) implies,
\begin{align}\label{eq:Opt2GradHD}
 \| \nabla d(\barlambda) -\nabla d(\lambda_*)\| \leq \sqrt{8mG^2\beps}~.
\end{align}
Now, using $\bg: = \bh(\bx)$ (Alg.~\ref{alg:MinimizingPrimalHighDim}), and $\|\bg-\nabla d(\barlambda)\|\leq \sqrt{mG^2\teps}$   (Lemma~\ref{lemma:GradOracleLemmaHD}), as well as $\nabla d(\lambda_*) = \bh(x^*)$  (Lemma~\ref{lemma:GradientLemmaHighDim}), we conclude from Eq.~\eqref{eq:Opt2GradHD}: 
\begin{align} \label{eq:GradCloseHD}
&
\|\bh(\bx) -  \bh(x^*)\| 
\leq \|\bh(\bx)  -  \nabla d(\barlambda)\| + \|\nabla d(\barlambda) - \bh(x^*)\|  
\nonumber\\
&
=\|\bg-\nabla d(\barlambda)\| + \| \nabla d(\barlambda) -\nabla d(\lambda_*)\| 
\leq
\sqrt{16 mG^2\beps}~,
\end{align}
where we used $\teps\leq \beps$.
The above implies that $\forall i\in[m]$, 
$
h_i(\bx)  =  h_i(x^*) + (h_i(\bx)-h_i(x^*)) 
\leq
 h_i(x^*) + |h_i(\bx)-h_i(x^*)| 
 \leq
 0 + \| \bh(\bx)-\bh(x^*)\|_\infty 
 \leq
 mG\sqrt{16\beps} 
 \leq
 \eps,$ 
 where the second inequality uses the feasibility of $x^*$, and the last line uses the definition of $ \beps$ (we assume $R\geq1$).
 This concludes the first part of the proof. Moreover, from Eq.~\eqref{eq:GradCloseHD} we also get,
\begin{align} \label{eq:KKT_HD}
&
-\barlambda^\top \bh(\bx)  
\nonumber\\
&
= 
-\barlambda^\top (\bh(\bx) -\bh(x^*)) - (\barlambda - \lambda_*)^\top \bh(x^*) -  (\lambda_*)^\top \bh(x^*) 
\nonumber\\
&
\leq
\sqrt{16mG^2\beps}\|\barlambda\| + \nabla d(\lambda_*)^\top (\lambda_*-\barlambda)  + 0 
\nonumber\\
&
\leq
 mG\sqrt{16\beps}\|\barlambda\|_\infty + d(\lambda_*) - d(\barlambda) 
\nonumber\\
&
\leq
mGR\sqrt{16\beps} + 4\beps 
\leq
5\eps~.
\end{align}
where the first inequality uses Eq.~\eqref{eq:GradCloseHD} as well as $\bh(x^*)  = \nabla d(\lambda_*)$ (Lemma~\ref{lemma:GradientLemmaHighDim}) and complementary slackness, which implies $(\lambda_*)^\top  \bh(x^*) = 0$; the second inequality  uses the concavity of $d(\cdot)$ implying that $d(\lambda_*) - d(\barlambda)  \geq \nabla d(\lambda_*)^\top (\lambda_*-\barlambda)$, and the last line uses the definition of $\beps$ as well as $\beps\leq \eps$. Using Eq.~\eqref{eq:KKT_HD} together with $\teps$-optimality of $\bx$ with respect to $\mathcal L(\cdot,\barlambda)$ (Alg.~\ref{alg:MinimizingPrimalHighDim}) implies that 
$\forall x\in\K: =\{ x: h_i(x)\leq 0;~\forall i\in[m] \}$ we have,
$
\|\bx-x_0\|^2 
\leq
 \|x-x_0\|^2 + \barlambda^\top \bh(x) -  \barlambda^\top \bh(\bx) +\teps
 \leq 
 \|x-x_0\|^2 + 6\eps~,
 $
and we used $\teps\leq \eps$, and
$\barlambda\geq 0, \bh(x)\leq0$.
This concludes  the proof.\\
\textbf{Runtime:} for a single $t\in[T]$ we invoke Alg.~\ref{alg:MinimizingPrimalHighDim}, and its runtime is $O(nm^{1.5}\log(m/\eps))$ (Lemma~\ref{lemma:GradOracleLemmaHD}), additionally $\tau_{CP}$ for the update. Multiplying this by $T$ we get a runtime of $O( n m^{2.5}{\theta^{-1}}\log^2(m/\eps)+\tau_{CP}m\theta^{-1}\log(m/\eps) )$.
Also, every call to the separation oracle for $\D$ takes $O(m)$ 
which is negligible 
compared to computing the gradient and value oracles.
\end{proof}
Note that inside our algorithm we could use not only Vaidya's and Ellipsoid methods, but any other cutting plane scheme. For example, the faster cutting plane methods proposed by \citet{lee2015faster}, \citet{jiang2020improved} can be used as well.

\section{EXPERIMENTAL EVALUATION}
\label{sec:Experiment}
\subsection{Synthetic Problem}
\looseness -1 \looseness -1 We first demonstrate the performance of our approach on synthetic problems of projection onto a randomly generated quadratic set and onto their intersection.
\begin{align}\label{eq:experiment}
\min_{x\in\R^n} & ~\|x-x_p\|^2\\
\text{subject to } & ~(x-x_i)^TA_i(x-x_i) \leq 0,~ i=1,\ldots,m.\nonumber
\end{align}
The matrices $A_i$ are generated randomly in such a way that they are positive definite and have norm equal to $1$.  We compare our approach with the Interior Point Method (IPM) from the MOSEK solver, as well as with SLSQP from the \textit{scipy.optimize.minimize} package. For Algorithm \ref{alg:MinimizingPrimalHighDim}, to solve the primal subproblems we use the AGD method as described before. We select the smoothness parameter $L$ is based on the norms of the matrices $A_i,$ and tune the parameter $R$ empirically using the doubling trick.
The run-times are shown in 
Table \ref{Table1}. 
The run-times are averaged over 5 runs of the method on the random inputs. The accuracy is fixed to $10^{-4}.$ 
\begin{table*}[!ht]\label{Table1}
    \centering
    \caption{Run-times (in seconds). Hereby, $n$ is the dimensionality of the problem, the number of constraints is $2$.} 
    \vspace{0.1 in}
    \begin{tabular}{|c|c|c|c|c|c|c|c|c|c|}
    \hline
     m = 2, n: & 10 & 100 & 500 & 1000 & 2000 & 5000 & 8000 & 10000 & 12000\\
    \hline
       SLSQP 
      & 0.011
      & 0.384
      & 3.481
      & 9.757
      & 47.143
      & 573.980
     & -
     & -
     & -
       \\
       IPM 
       & 0.059
       & 0.073
       & 0.577
       & 2.427
       & 11.850
       & 118.414
       & 408.137
       & 751.216
       & 901.878
    \\
          Fast Proj 
      & 0.416 
      & 2.429 
      & 3.746 
      & 15.504
      & 22.482 
      & 141.704 
      & 350.681 
      & 547.240
      & 666.231
       \\
       ADMM 
       & 23.761 
       & 92.836
       & 285.383
       & - 
       & - 
       & - 
       & - 
       & -
       & -\\
       \hline
    \end{tabular}
    \label{tab:runtimes}
\end{table*}
The results demonstrate a substantial performance improvement obtained by our fast projection approach as the dimensionality increases. The runtime in seconds is not a perfect performance measure, but is the most reasonable measure we could think of. Comparing the number of iterates hides the complexity of each iteration which might be huge for interior point methods.

\subsection{Learning the Kernel Matrix in Discriminant Analysis via QCQP \citep{kim2006optimal, ye2007learning,  basu2017large}} 
We next consider an application in multiple kernel learning. Consider a standard binary classification setup where
 $\X$ -- a subset of $\R^n$ -- denotes the input space, and $\Y = \{-1,+1\}$ denotes the output (class label) space. 
We assume that the examples are independently drawn from a fixed unknown probability distribution over $\X\times\Y.$
We model our data with positive definite {\em kernel} functions \cite{ScholkopfBook}.
In particular, for any $x_1,\ldots,x_n \in \X,$ the \textit{Gram} matrix, defined by
 $G_{jk} = K(x_j,x_k)$ is positive semi-definite.  
Let $X = [x_1^+,\ldots,x_{n_+}^+,x_1^-,\ldots,x_{n_-}^-]$ be a data matrix of size $n = n_+ + n_- $, where $\{x_1^+,\ldots,x_{n_+}^+\}$ and $\{x_1^-,\ldots,x_{n_-}^-\}$ are the data points from positive and negative classes.  For binary classification, the problem of kernel learning for discriminant analysis seeks, given a set of $p$ kernel matrices $G^i = K^i(x_j,x_k), x_j,x_k\in X, i\in[p], G^i\in\R^{n\times n}$ to learn an optimal linear combination $G\in \mathcal G = \left\{ G\,|\, G = \sum_{i=1}^{p}\theta_i G^i, \sum_{i=1}^{p}\theta_i  =  1, \theta_i\geq 0 \right\}$.
This problem was introduced by \citet{fung2004fast}, reformulated as an SDP by \citet{kim2006optimal}, and as a much more tractable QCQP by \citet{ye2007learning}. 
 Latter approach learns an optimal kernel matrix $\tilde G \in \tilde {\mathcal G} = \left\{\tilde G\,|\,\tilde G = \sum_{i=1}^{p}\theta_i\tilde G^i, \sum_{i=1}^{p}\theta_i r_i  =  1, \theta_i\geq 0 \right\},$ 
where $\tilde G^i = G^iPG^i, r_i = \text{Trace}(\tilde G^i)$, $P = I - \frac{1}{n} \textbf{1}_n \textbf{1}_n^T,$ and $\textbf{1}_n$ is the vector of all ones of size $n$,  
 by solving the following convex QCQP
 \begin{align}\label{problem:kernel}
 \max_{\beta, t} & ~-\frac{1}{4} \beta^T\beta + \beta^T a -\frac{\lambda}{4}t\\
 \text{subject to }&~ t \geq \frac{1}{r_i}\beta^T\tilde G^i\beta,~ i = 1,\ldots, p,
 \end{align}
\looseness -1 where
$a = [1/n^+, \ldots, 1/n^+,-1/n^-, \ldots, -1/n^-]\in \R^n, \beta \in \R^n$. Hereby $\lambda$ is a regularization parameter that we set to $\lambda = 10^{-4}.$ The optimal $\theta$ corresponds to the dual solution of the above problem (\ref{problem:kernel}). Note that in this application, the number of data points $n$ is much larger than the number of constraints (i.e., the number of kernel matrices), making it ideally suited for our approach.
We run our algorithm applied for this problem over $\beta$ with fixed $t = 5\cdot 10^{-8}$.    Then the problem becomes strongly convex: 
$
        \arg\max_{\beta} -\frac{1}{4}\beta^T\beta + \beta^T a + \lambda^T t = \arg\max_{\beta} -\frac{1}{4}(\beta^T\beta  - 4\beta^T a + 4a^Ta) = \arg\max_{\beta} -\frac{1}{4}\|\beta - 2a\|^2_2.
    $
We use the \textit{doc-rna}  dataset \cite{doc_rna} from LIBSVM  with $n=4000, 10000, 11000$ data points and compare the results and the running time with the IPM.  We focus on learning a convex combination of $m$ Gaussian Kernels $K(x,z) = \sum_{i = 1}^{m} \theta_i e^{-\|x-z\|^2/\sigma_i^2}$ with different bandwidth parameters $\sigma_i$, chosen uniformly on the logarithmic scale over the interval $[10^{-1}, 10^2],$ as in \cite{kim2006optimal, ye2007learning}. Results are shown in  Table \ref{Table2} below.
\begin{table}[ht]\label{Table2}
    \centering
    \caption{Run-times (in seconds). Hereby, $n$ is the number of data points  (dimensionality), the number of kernels is $3$ (number of constraints). For large problems, our approach outperforms IPM.}
    \vspace{0.1 in}
    \begin{tabular}{|c|c|c|c|c|}
    \hline
     m = 3, n: & 4000 & 10000 & 11000\\
    \hline
        Fast Proj 
        & 230.281
        & 768.086
        & 1216.9440
       \\
        IPM 
        & 75.133
        & 906.631
        & 1302.088
        \\
       \hline
    \end{tabular} 
    \label{tab:runtimes}
\end{table}
Moreover, for the Kernel Learning problem with
$\teps = 500 \eps^2$,  we present the results for IPM and Fast Projection algorithms for $ m = 3, n= 11000$ dependent on the target accuracy in Table 3.
\begin{table}[!ht]
    \centering
    \caption{Run-times (in seconds). Hereby, $\varepsilon$ is the target accuracy in objective value, the number of kernels is $3$ (number of constraints). For large problems and smaller accuracies, our approach outperforms IPM. }
    \vspace{0.1 in}
    \begin{tabular}{|c|c|c|c|}
    \hline
     $\eps $ & $ 10^{-6}$ & $10^{-7}$ & $10^{-8}$  \\
    \hline
      Fast Proj 
      & 83.2381 & 519.7166 & 1216.9440\\
       IPM 
      & 1011.5410 & 1070.3363 & 1302.0882\\
       \hline
    \end{tabular}
        \label{tab:runtimes}
\end{table}
Note that quadratic constraints do not satisfy Lipschitz continuity assumption on the whole $\R^n$.  However, the Lipschitz  continuity holds on any compact set inside $\R^n$. Since the AGD algorithm keeps the iterates on the compact set, this is enough to guarantee the Lipschitz continuity. Moreover, the Lipschitz constant $G$ itself is needed only to specify the accuracy for AGD $\teps.$ It only affects the runtime of AGD logarithmically. The parameter $H$ is not needed to be known since it influences only the upper bound on the runtime of the ellipsoid method. 
\footnote{The experiments were run on a machine with Intel Core i7-7700K, 64Gb RAM.}

\section{CONCLUSION}
We proposed a novel method for fast projection onto smooth convex constraints.
We employ a primal-dual approach, and combine cutting plane schemes with Nesterov's accelerated gradient descent. We analyze its performance and prove its effectiveness in  high-dimensional settings with a small number of constraints. The results are generalizable to any strongly-convex objective with smooth convex constraints. 
Our work demonstrates applicability of cutting plane algorithms in the field of Machine Learning and can potentially improve efficiency of solving  high dimensional constrained optimization problems. Enforcing constraints can be of crucial importance when ensuring reliability and safety of machine learning systems. 

\section*{Acknowledgements}
We thank the reviewers for the helpful comments. This project received funding from the European Research Council (ERC) under the European Union’s Horizon 2020 research and innovation programme
grant agreement No 815943, the Swiss National Science Foundation under the grant SNSF 200021\_172781, and under NCCR Automation under grant agreement 51NF40 180545, as well as the Israel Science Foundation (grant No. 447/20).
\looseness=-1


\addcontentsline{toc}{section}{Bibliography}
\bibliographystyle{apalike}

\bibliography{bibliography}

\begin{thebibliography}{}

\bibitem[Aholt et~al., 2012]{aholt2012qcqp}
Aholt, C., Agarwal, S., and Thomas, R. (2012).
\newblock A qcqp approach to triangulation.
\newblock In {\em European Conference on Computer Vision}, pages 654--667.
  Springer.

\bibitem[Allen-Zhu et~al., 2017]{allen2017linear}
Allen-Zhu, Z., Hazan, E., Hu, W., and Li, Y. (2017).
\newblock Linear convergence of a frank-wolfe type algorithm over trace-norm
  balls.
\newblock In {\em Advances in Neural Information Processing Systems}, pages
  6192--6201.

\bibitem[Altman and Asingleutility, 1999]{Altman99constrainedmarkov}
Altman, E. and Asingleutility, I. (1999).
\newblock Constrained markov decision processes.

\bibitem[Arora et~al., 2005]{arora2005fast}
Arora, S., Hazan, E., and Kale, S. (2005).
\newblock Fast algorithms for approximate semidefinite programming using the
  multiplicative weights update method.
\newblock In {\em 46th Annual IEEE Symposium on Foundations of Computer Science
  (FOCS'05)}, pages 339--348. IEEE.

\bibitem[Basu et~al., 2017]{basu2017large}
Basu, K., Saha, A., and Chatterjee, S. (2017).
\newblock Large-scale quadratically constrained quadratic program via
  low-discrepancy sequences.
\newblock In {\em Advances in Neural Information Processing Systems}, pages
  2297--2307.

\bibitem[Boyd et~al., 2011]{boyd2011distributed}
Boyd, S., Parikh, N., Chu, E., Peleato, B., Eckstein, J., et~al. (2011).
\newblock Distributed optimization and statistical learning via the alternating
  direction method of multipliers.
\newblock {\em Foundations and Trends{\textregistered} in Machine learning},
  3(1):1--122.

\bibitem[Condat, 2016]{condat2016fast}
Condat, L. (2016).
\newblock Fast projection onto the simplex and the $\ell_1$ ball.
\newblock {\em Mathematical Programming}, 158(1):575--585.

\bibitem[Eckstein and Yao, 2012]{eckstein2012augmented}
Eckstein, J. and Yao, W. (2012).
\newblock Augmented lagrangian and alternating direction methods for convex
  optimization: A tutorial and some illustrative computational results.
\newblock {\em RUTCOR Research Reports}, 32(3).

\bibitem[Frank and Wolfe, 1956]{frank1956algorithm}
Frank, M. and Wolfe, P. (1956).
\newblock An algorithm for quadratic programming.
\newblock {\em Naval research logistics quarterly}, 3(1-2):95--110.

\bibitem[Fung et~al., 2004]{fung2004fast}
Fung, G., Dundar, M., Bi, J., and Rao, B. (2004).
\newblock A fast iterative algorithm for fisher discriminant using
  heterogeneous kernels.
\newblock In {\em Proceedings of the twenty-first international conference on
  Machine learning}, page~40.

\bibitem[Garber, 2016]{garber2016faster}
Garber, D. (2016).
\newblock Faster projection-free convex optimization over the spectrahedron.
\newblock In {\em Advances in Neural Information Processing Systems}, pages
  874--882.

\bibitem[Garber and Hazan, 2013]{garber2013playing}
Garber, D. and Hazan, E. (2013).
\newblock Playing non-linear games with linear oracles.
\newblock In {\em 2013 IEEE 54th Annual Symposium on Foundations of Computer
  Science}, pages 420--428. IEEE.

\bibitem[Garber and Hazan, 2015]{garber2015faster}
Garber, D. and Hazan, E. (2015).
\newblock Faster rates for the frank-wolfe method over strongly-convex sets.
\newblock In {\em International Conference on Machine Learning}, pages
  541--549.

\bibitem[Garber and Meshi, 2016]{garber2016linear}
Garber, D. and Meshi, O. (2016).
\newblock Linear-memory and decomposition-invariant linearly convergent
  conditional gradient algorithm for structured polytopes.
\newblock In {\em Advances in Neural Information Processing Systems}, pages
  1001--1009.

\bibitem[Giselsson and Boyd, 2014]{Giselsson2014MetricSI}
Giselsson, P. and Boyd, S.~P. (2014).
\newblock Metric selection in douglas-rachford splitting and admm.
\newblock {\em arXiv: Optimization and Control}.

\bibitem[Goldstein et~al., 2014]{goldstein2014fast}
Goldstein, T., O'Donoghue, B., Setzer, S., and Baraniuk, R. (2014).
\newblock Fast alternating direction optimization methods.
\newblock {\em SIAM Journal on Imaging Sciences}, 7(3):1588--1623.

\bibitem[Gustavo et~al., 2018]{gustavo2018fast}
Gustavo, C., Wohlberg, B., and Rodriguez, P. (2018).
\newblock Fast projection onto the $\ell_{\infty},\ell_1$-mixed norm ball using
  steffensen root search.
\newblock In {\em 2018 IEEE International Conference on Acoustics, Speech and
  Signal Processing (ICASSP)}, pages 4694--4698. IEEE.

\bibitem[He and Yuan, 2012]{he20121}
He, B. and Yuan, X. (2012).
\newblock On the o(1/n) convergence rate of the douglas--rachford alternating
  direction method.
\newblock {\em SIAM Journal on Numerical Analysis}, 50(2):700--709.

\bibitem[Hong and Luo, 2017]{Hong2017OnTL}
Hong, M. and Luo, Z. (2017).
\newblock On the linear convergence of the alternating direction method of
  multipliers.
\newblock {\em Mathematical Programming}, 162:165--199.

\bibitem[Huang and Palomar, 2014]{huang2014randomized}
Huang, Y. and Palomar, D.~P. (2014).
\newblock Randomized algorithms for optimal solutions of double-sided qcqp with
  applications in signal processing.
\newblock {\em IEEE Transactions on Signal Processing}, 62(5):1093--1108.

\bibitem[Iudin and Nemirovskii, 1977]{iudin1977evaluation}
Iudin, D. and Nemirovskii, A. (1977).
\newblock Evaluation of informational complexity of mathematical-programming
  programs.
\newblock {\em Matekon}, 13(2):3--25.

\bibitem[Jaggi, 2013]{jaggi2013revisiting}
Jaggi, M. (2013).
\newblock Revisiting frank-wolfe: Projection-free sparse convex optimization.
\newblock In {\em ICML (1)}, pages 427--435.

\bibitem[Jiang et~al., 2020]{jiang2020improved}
Jiang, H., Lee, Y.~T., Song, Z., and Wong, S. C.-w. (2020).
\newblock An improved cutting plane method for convex optimization,
  convex-concave games, and its applications.
\newblock In {\em Proceedings of the 52nd Annual ACM SIGACT Symposium on Theory
  of Computing}, pages 944--953.

\bibitem[Jin and Sidford, 2020]{jin2020efficiently}
Jin, Y. and Sidford, A. (2020).
\newblock Efficiently solving mdps with stochastic mirror descent.
\newblock In {\em International Conference on Machine Learning}, pages
  4890--4900. PMLR.

\bibitem[Juditsky, 2015]{JuditskyBook}
Juditsky, A. (2015).
\newblock Convex optimization ii: Algorithms.
\newblock
  \url{https://ljk.imag.fr/membres/Anatoli.Iouditski/cours/convex/chapitre_22.pdf}.

\bibitem[Karmarkar, 1984]{karmarkar1984new}
Karmarkar, N. (1984).
\newblock A new polynomial-time algorithm for linear programming.
\newblock In {\em Proceedings of the sixteenth annual ACM symposium on Theory
  of computing}, pages 302--311.

\bibitem[Kim et~al., 2006]{kim2006optimal}
Kim, S.-J., Magnani, A., and Boyd, S. (2006).
\newblock Optimal kernel selection in kernel fisher discriminant analysis.
\newblock In {\em Proceedings of the 23rd international conference on Machine
  learning}, pages 465--472.

\bibitem[Lacoste-Julien and Jaggi, 2015]{lacoste2015global}
Lacoste-Julien, S. and Jaggi, M. (2015).
\newblock On the global linear convergence of frank-wolfe optimization
  variants.
\newblock In {\em Advances in Neural Information Processing Systems}, pages
  496--504.

\bibitem[Lan et~al., 2017]{lan2017conditional}
Lan, G., Pokutta, S., Zhou, Y., and Zink, D. (2017).
\newblock Conditional accelerated lazy stochastic gradient descent.
\newblock In {\em International Conference on Machine Learning}, pages
  1965--1974.

\bibitem[Lan and Zhou, 2016]{lan2016conditional}
Lan, G. and Zhou, Y. (2016).
\newblock Conditional gradient sliding for convex optimization.
\newblock {\em SIAM Journal on Optimization}, 26(2):1379--1409.

\bibitem[Lee et~al., 2015]{lee2015faster}
Lee, Y.~T., Sidford, A., and Wong, S. C.-w. (2015).
\newblock A faster cutting plane method and its implications for combinatorial
  and convex optimization.
\newblock In {\em 2015 IEEE 56th Annual Symposium on Foundations of Computer
  Science}, pages 1049--1065. IEEE.

\bibitem[Levin, 1965]{levin1965algorithm}
Levin, A.~Y. (1965).
\newblock An algorithm for minimizing convex functions.
\newblock In {\em Doklady Akademii Nauk}, volume 160, pages 1244--1247. Russian
  Academy of Sciences.

\bibitem[Levy and Krause, 2019]{levy19projection}
Levy, K.~Y. and Krause, A. (2019).
\newblock Projection free online learning over smooth sets.
\newblock In {\em Proc. International Conference on Artificial Intelligence and
  Statistics (AISTATS)}.

\bibitem[Li and Li, 2020]{li2020fast}
Li, Q. and Li, X. (2020).
\newblock Fast projection onto the ordered weighted $\ell_1$ norm ball.
\newblock {\em arXiv preprint arXiv:2002.05004}.

\bibitem[{Li} et~al., 2020]{li2020fully}
{Li}, Y., {Cao}, X., and {Chen}, H. (2020).
\newblock Fully projection-free proximal stochastic gradient method with
  optimal convergence rates.
\newblock {\em IEEE Access}, 8:165904--165912.

\bibitem[Mahdavi et~al., 2012]{mahdavi2012stochastic}
Mahdavi, M., Yang, T., Jin, R., Zhu, S., and Yi, J. (2012).
\newblock Stochastic gradient descent with only one projection.
\newblock In {\em Advances in Neural Information Processing Systems}, pages
  494--502.

\bibitem[Nemirovski and Todd, 2008]{nemirovski2008interior}
Nemirovski, A.~S. and Todd, M.~J. (2008).
\newblock Interior-point methods for optimization.
\newblock {\em Acta Numerica}, 17:191--234.

\bibitem[Nesterov, 1998]{nesterov1998introductory}
Nesterov, Y. (1998).
\newblock Introductory lectures on convex programming volume i: Basic course.

\bibitem[Newman, 1965]{newman1965location}
Newman, D.~J. (1965).
\newblock Location of the maximum on unimodal surfaces.
\newblock {\em Journal of the ACM (JACM)}, 12(3):395--398.

\bibitem[Nishihara et~al., 2015]{nishihara2015general}
Nishihara, R., Lessard, L., Recht, B., Packard, A., and Jordan, M.~I. (2015).
\newblock A general analysis of the convergence of admm.
\newblock {\em arXiv preprint arXiv:1502.02009}.

\bibitem[Peters and Herrmann, 2019]{peters2019algorithms}
Peters, B. and Herrmann, F.~J. (2019).
\newblock Algorithms and software for projections onto intersections of convex
  and non-convex sets with applications to inverse problems.
\newblock {\em arXiv preprint arXiv:1902.09699}.

\bibitem[Plotkin et~al., 1995]{Plotkin95fastapproximation}
Plotkin, S.~A., Shmoys, D.~B., and Éva Tardos (1995).
\newblock Fast approximation algorithms for fractional packing and covering
  problems.

\bibitem[Schmidt et~al., 2009]{schmidt2009optimizing}
Schmidt, M., Berg, E., Friedlander, M., and Murphy, K. (2009).
\newblock Optimizing costly functions with simple constraints: A limited-memory
  projected quasi-newton algorithm.
\newblock In {\em Artificial Intelligence and Statistics}, pages 456--463.

\bibitem[Schölkopf et~al., 2018]{ScholkopfBook}
Schölkopf, B., Smola, A.~J., and Bach, F. (2018).
\newblock {\em Learning with Kernels: Support Vector Machines, Regularization,
  Optimization, and Beyond}.
\newblock The MIT Press.

\bibitem[Shor, 1977]{shor1977cut}
Shor, N.~Z. (1977).
\newblock Cut-off method with space extension in convex programming problems.
\newblock {\em Cybernetics}, 13(1):94--96.

\bibitem[Uzilov et~al., 2006]{doc_rna}
Uzilov, A.~V., Keegan, J.~M., and Mathews, D.~H. (2006).
\newblock Detection of non-coding rnas on the basis of predicted secondary
  structure formation free energy change.
\newblock {\em BMC Bioinformatics, 7(173)}.

\bibitem[Vaidya, 1989]{vaidya1989new}
Vaidya, P.~M. (1989).
\newblock A new algorithm for minimizing convex functions over convex sets.
\newblock In {\em 30th Annual Symposium on Foundations of Computer Science},
  pages 338--343. IEEE.

\bibitem[Xu et~al., 2017]{xu17c}
Xu, Z., Taylor, G., Li, H., Figueiredo, M. A.~T., Yuan, X., and Goldstein, T.
  (2017).
\newblock Adaptive consensus {ADMM} for distributed optimization.
\newblock In Precup, D. and Teh, Y.~W., editors, {\em Proceedings of the 34th
  International Conference on Machine Learning}, volume~70 of {\em Proceedings
  of Machine Learning Research}, pages 3841--3850. PMLR.

\bibitem[Yang et~al., 2017]{yang2017richer}
Yang, T., Lin, Q., and Zhang, L. (2017).
\newblock A richer theory of convex constrained optimization with reduced
  projections and improved rates.
\newblock In {\em Proceedings of the 34th International Conference on Machine
  Learning-Volume 70}, pages 3901--3910. JMLR. org.

\bibitem[Ye et~al., 2007]{ye2007learning}
Ye, J., Ji, S., and Chen, J. (2007).
\newblock Learning the kernel matrix in discriminant analysis via quadratically
  constrained quadratic programming.
\newblock In {\em Proceedings of the 13th ACM SIGKDD international conference
  on Knowledge discovery and data mining}, pages 854--863.

\bibitem[Zhu et~al., 2006]{zhu2006graph}
Zhu, X., Kandola, J., Lafferty, J., and Ghahramani, Z. (2006).
\newblock Graph kernels by spectral transforms.
\newblock {\em Semi-supervised learning}, pages 277--291.

\end{thebibliography}


\newpage

\appendix
\onecolumn

\section{Single constraint case   analysis}\label{appendix:A}
\paragraph{Formal description:}
Our method for the case of a single constraint is described in Alg.~\ref{alg:MinimizingPrimal}, and Alg.~\ref{alg:1DimConc} .
We can think of Alg.~\ref{alg:MinimizingPrimal} , as an oracle $\O: \reals_{+}\mapsto \reals^n\times \reals$, that receives $\lambda\geq 0$, and outputs $(x,g,v)$ such that $g$ is an $\tilde\eps$-accurate estimate of the $\nabla d(\lambda)$.
Alg.~\ref{alg:1DimConc} invokes the oracle $\O$ in every round, and uses the resulting gradient estimates in order to solve the dual problem using one-dimensional bisection. Finally, Alg.~\ref{alg:1DimConc}  chooses the  iterate $\lambda_t$ that approximately maximizes the dual problem, and outputs the corresponding $x_t\in \reals^n$ as the solution to the original projection problem (Eq.~\eqref{eq:problem_1}).

\begin{minipage}[c]{0.48\linewidth}
\vspace{-0.8 cm}
\begin{algorithm}[H]
\caption{$\O$-an approximate oracle for $\nabla d(\cdot)$ }
\label{alg:MinimizingPrimal}
\begin{algorithmic}
\STATE \textbf{Input}:  $\lambda\geq 0$,  target accuracy $\teps$
\STATE Compute $x_{\lambda}$, an $\teps$-optimal solution of,
$$
\min_{x\in\reals^n} L(x,\lambda): = \|x-x_0\|^2 + \lambda h(x)~.
$$
\STATE \textbf{Method:}  Nesterov's AGD (Alg.~\ref{alg:Nes}) with\\ $\alpha=2$,
~ $\beta = 2+\lambda L$, \\and~
$T = O(\sqrt{\beta}\log(\beta B/\teps))$~.
\STATE \textbf{Let}:~~$v:= \|x_{\lambda}-x_0\|^2 + \lambda  h(x_{\lambda})$, \\~~~~~~~~
$g: = h(x_{\lambda})$
\STATE \textbf{Output}: $(x_{\lambda}, g,v)$
\end{algorithmic}
\end{algorithm}
\end{minipage}
\begin{minipage}[c]{0.48\linewidth}
\begin{algorithm}[H]
\caption{
Bisection
}
\label{alg:1DimConc}
\begin{algorithmic}
\STATE \textbf{Input}:  $\eps.$ { Set}: $\lambda_{\min} =0$,  $\lambda_{\max} = R$,
\STATE  $\tilde{\eps} 
$ as in Thm. \ref{theorem:Main}, $T = O( \log_2(R G/\eps))$
\FOR{$t=1 \ldots T$ }
\STATE {Calculate:} $\lambda_t  = (\lambda_{\max} + \lambda_{\min})/2$
\STATE {Calculate:} $(x_t, g_t,v_t) \gets \O(\lambda_t,\tilde{\eps})$ using Alg.\ref{alg:MinimizingPrimal}
\IF{$g_t> 0$}
\STATE   {Update:}  $\lambda_{\min}\gets \lambda_t$
\ELSE   
\STATE {Update:}  $\lambda_{\max}\gets \lambda_t$
\ENDIF
\ENDFOR
\STATE Set $\tau = \arg\max_{t\in[T]}v_t$
\STATE \textbf{Output}: $(x_{\tau}, \lambda_{\tau})$
\end{algorithmic}
\end{algorithm}
\end{minipage}


\subsection{Using Nesterov's Method within Alg.~\ref{alg:MinimizingPrimal}, Alg.~ \ref{alg:MinimizingPrimalHighDim}}
\looseness -1 Alg.~\ref{alg:MinimizingPrimal} requires a fast approximate method to solving $\min_{x\in\reals^n}L(x,\lambda)$.
This  is a $2$-strongly-convex and $(2+\|\lambda\|_1 L)$-smooth \emph{unconstrained} problem, which suggests using fast first order method of Nesterov to do so. The benefit here is that the runtime is only linear in $n$ and logarithmic in the accuracy $\teps$.
Next we describe Nesterov's method and its guarantees.
\begin{theorem}[Nesterov \cite{nesterov1998introductory}]\label{thm:nesterov}
\looseness -1 Let $F$ be $\alpha$-strongly-convex and $\beta$-smooth with an optimal solution $x^*$, and let $\kappa:= \beta/\alpha$. Then applying Nesterov's AGD with $T = O(\sqrt{\kappa}\log(\beta\|x_0-x^*\|^2/\eps)$ ensures,
$$
F(y_T) - F(x^*) \leq \eps~.
$$ 
\end{theorem}

\newpage
\section{Handling the Bound on $R$}
\subsection{Obtaining a Bound on $R$ in Two Special Cases}\label{appendix:B}
\subsubsection{Case 1: Single Constraint}
Assume that we are in the case of a single smooth constraint $h(\cdot)$.
In this case it is common and natural to assume that 
 that the gradient of the constraint is lower bounded on the boundary i.e., $\min_{x:h(x)=0}\|\nabla h(x) \| \geq Q$  (see e.g.,~\cite{mahdavi2012stochastic,yang2017richer,levy19projection}). 
  This immediately implies a bound on $\lambda_*$ as shown in the next lemma, 
  \begin{lemma}\label{lemma:OndimBoundDual}
Assume that $\min_{x:h(x)=0}\|\nabla h(x) \| \geq Q$, and that  $\min_{x: h(x)\leq 0} \|x-x_0\| \leq B$. Then the optimal solution to the dual problem $\lambda_*$ is bounded, $\lambda_*\leq R: =2B/Q$.
\end{lemma}
\begin{proof}
By definition $\lambda_*\geq 0$.
If $h(x^*)<0$ complementary slackness implies $\lambda_*=0$, and the bound holds.
If $h(x^*)=0$, then using the optimality of $x^*$ and $\lambda_*$ implies,
$$
2(x^*-x_0) = -\lambda_* \nabla h(x^*)
$$
Implying that,
$$
|\lambda_*| = \frac{2\|x^*-x_0\|}{\|\nabla h(x^*)\|} \leq R: =\frac{2B}{Q}
$$
\end{proof}

\subsubsection{Case 2: Quadratic Constraints}
Assume we have multiple quadratic constraints, $h_i(x): = x^\top A_i x - c_i\leq 0~,\forall i\in[m]$. Where $A_i$'s are PSD and $c_i$'s are positive. In this case optimality implies that,
$$
\sum_{i=1}^m \lambda_*^{(i)} A_i x^* = x_0-x^*
$$
Multiplying both sides by $x^*$ and using complementary slackness gives,
$$
\sum_{i=1}^m \lambda_*^{(i)} c_i = (x_0-x^*)^\top x^*
$$
Since $\lambda_*^{(i)}$ and $c_i$'s are non-negative this implies that $\forall i\in[m]$,
$$
\lambda_*^{(i)} = \frac{1}{c_i}\lambda_*^{(i)} c_i \leq  \frac{1}{c_i}\sum_{i=1}^m \lambda_*^{(i)} c_i = \frac{1}{c_i} (x_0-x^*)^\top x^* \leq \frac{1}{c_i} B\cdot X^*~,
$$
where $X^*$ bounds the norm of $x^*$, and $B$ satisfies  $\min_{x: h(x)\leq 0} \|x-x_0\| \leq B$. Thus in this case we can take,
$$
R: = \max_{i\in[m]} c_i^{-1}BX^*~.
$$



\subsection{A Generic Way to Handling $R$}
\label{appendix:B2}
An effective and practical strategy to choosing $R$ both in theory and in practice is to  use a standard doubling trick to estimate $R$ ``on the fly", which will only cause a factor of $\sqrt{2}\ln 2$  increase in the total running time. 

The idea is the following: Underestimation of $R$ can be detected by the convergence of the dual solution to the boundary of the dual $\ell_{\infty}$-box with radius $R$. Consequently, one can start with $R_0=1$, and then double the estimation $R$ until the resulting solution does not intersect with the boundary of this $\ell_{\infty}$-box. In the worst case one will only have to increase the estimate $R$ by no more than a  logarithmic number of times,i.e.,no more than $1+\log_2 R$ times. Now note that given a fixed estimation $a$ of the radius $R$,  the total runtime is proportional to $\sqrt{a}\log a$, where the $\log a$ comes from the cutting plane method, and the $\sqrt{a}$ factor comes from the number of the internal iterations $T$ that is defined  inside Alg.~\ref{alg:MinimizingPrimalHighDim}. Thus, upon using doubling, the total runtime increases by no more than a constant factor.

\newpage
\section{Duality of Projections} \label{appendix:DualProj}

Recall than in Section~\ref{sec:DualProject} our goal is to solve the following problem,
\begin{align}\label{eq:problem_NormBallAppendix}
  &\min_{ x\in \R^n:P(x) \leq 1} ~~~~~~~\|x_0  - x\|^2~.
\end{align}
And we assume that we have an oracle $\Pi_*$ that enables to project onto its dual norm ball, i.e., given $y\in\reals^n$, we can compute,
$$
\Pi_*(y): = \argmin_{x\in \R^n:P_*(x) \leq 1}\|y  - x\|^2
$$
where $P_*(\cdot)$ is the dual norm of $P(\cdot)$.

Our main statement (Theorem~\ref{thm:DualProj}) asserts that using $O(\log(1/\eps))$ calls to $\Pi_*(\cdot)$ we can solve Problem~\eqref{eq:problem_NormBallAppendix} up to an $\eps$-approximation, which establishes an efficient conversion  between projection onto norms and their duals.

\textbf{Our approach:} 
 Our starting point is similar to the Lagrangian formulation that we describe in Section~\ref{sec:SingleConstraint}.
 Concretely, the Lagrangian of Problem~\eqref{eq:problem_NormBallAppendix} is,
$$
        \mathcal L(x,\lambda): =\|x_0  - x\|^2  + \lambda (P(x)-1) ~.
$$
Now given $\lambda\geq0$, the dual problem is defined as follows,
$$
d(\lambda) = \min_{x\in\reals^n}\{ \|x_0  - x\|^2  + \lambda (P(x)-1) \}
$$
Now, the key idea is that we can compute the \emph{exact} gradients of $d(\lambda)$ using the dual oracle $\Pi_*(\cdot)$
(which is different from what we do in Sections~\ref{sec:SingleConstraint}, and \ref{sec:HighDimcase}). To see that, we will use the fact that the dual of $P_*(\cdot)$ is $P(\cdot)$, and therefore we can write $P(x) = \max_{y:P_*(y)\leq 1} y^\top x$. Plugging this back into the expression for  $d(\lambda)$ gives,
\begin{align} \label{eq:d_x_y_dual}
    d(\lambda) 
    &= 
    \min_{x\in\reals^n}\max_{y:P_*(y)\leq 1}\{ \|x_0  - x\|^2  + \lambda (x^\top y-1) \} \nonumber\\
    &=
    \max_{y:P_*(y)\leq 1}\min_{x\in\reals^n}\{ \|x_0  - x\|^2  + \lambda (x^\top y-1) \} \nonumber\\
    &=
    \max_{y:P_*(y)\leq 1}\{-\|(\lambda y/2)-x_0\|^2 + \|x_0\|^2 -\lambda \} \nonumber\\
    &= 
    -\lambda +  \|x_0\|^2 - \frac{\lambda^2}{4} \min_{y:P_*(y)\leq 1} \left\|y-\frac{2x_0}{\lambda}\right\|^2 
\end{align}
where the second line follows from strong-duality due to the fact that  $M(x,y):= \|x_0  - x\|^2  + \lambda (x^\top y-1)$ is convex in $x$ and concave in $y$; the third line follows since we have a closed form solution to the internal minimization problem for which 
$x_{\text{opt}} = x_0 - \frac{\lambda}{2}y $, and plugging this expression to obtain the minimal value.

 The above implies that the following is a sub-gradient (sub-derivative) of $d(\cdot)$ for any $\lambda>0$,
\begin{align}\label{eq:DualExactGrad}
\nabla d(\lambda) = -1 - \frac{\lambda}{2}\left \|\Pi_*\left(\frac{2x_0}{\lambda}\right)-\frac{2x_0}{\lambda} \right\|^2
+\left(\Pi_*\left(\frac{2x_0}{\lambda}\right)-\frac{2x_0}{\lambda}\right)^\top x_0
 \end{align}
where we have used the definition of $\Pi_*$.
Thus, Eq.~\eqref{eq:DualExactGrad} shows that we can compute the exact gradients of $d(\cdot)$ using the dual projection oracle $\Pi_*$.

Now, recall that  $d(\cdot)$ is a concave one-dimensional functions, which implies that given an exact gradient oracle to $d(\cdot)$ we can find an $\eps$-optimal solution to $\max_{\lambda \in [0,R]}d(\lambda)$ within $O(R\log(1/\eps))$ calls to the gradient oracle, by using concave bisection algorithm , \citep{JuditskyBook}, that appears in Alg.~\ref{alg:1DimConc} (recall that $R>0$ is a bound on the optimal value of $\lambda^*:= \max_{\lambda \geq 0}d(\lambda)$).
Concretely, after $O(R\log(1/\eps))$ calls to the exact gradient oracle of $d(\cdot)$ we can find a solution $\bar{\lambda}$ such that,
\begin{align}\label{eq:DualGuaranteesNormBall}
d(\lambda^*)-d(\bar{\lambda}) \leq \eps~.
\end{align}
Note that one-dimensional concave bisection is a private case of the more general cutting plane scheme we present in Alg.~\ref{alg:CuttingPlane}.

\textbf{Remark:} In contrast to what we do  in Alg.~\ref{alg:1DimConc}, where we invoke 
Alg.~\ref{alg:MinimizingPrimal} in order to devise \emph{approximate} gradient, value, and primal solution oracles for $d(\cdot)$; here we use $\Pi_*$ in order to devise exact oracles. For completeness we depict our exact oracle  in Alg.~\ref{alg:DualProj}.

\begin{algorithm}[H]
\caption{$\O$-an Exact oracle for $\nabla d(\cdot)$ }
\label{alg:DualProj}
\begin{algorithmic}
\STATE \textbf{Input}:  $\lambda\geq 0$, projection oracle onto dual norm ball $\Pi_*(\cdot)$
\STATE Compute $x_{\lambda}$, an exact solution of,
$$
\min_{x\in\reals^n} L(x,\lambda): = \|x-x_0\|^2 + \lambda (P(x)-1)~.
$$
\STATE \textbf{Method:}   
\STATE \textbf{Compute}~~~$y_\lambda: = \Pi_*\left( \frac{2x_0}{\lambda}\right)$
\STATE \textbf{Let}
$$
v = -\lambda - \frac{\lambda^2}{4}  \left\|y_\lambda-\frac{2x_0}{\lambda}\right\|^2 
~;~~~
g = -1 - \frac{\lambda}{2}\left \|y_\lambda-\frac{2x_0}{\lambda} \right\|^2
+\left(y_\lambda-\frac{2x_0}{\lambda}\right)^\top x_0
~;\qquad\text{\% see Equations~\eqref{eq:d_x_y_dual}, \eqref{eq:DualExactGrad}}
$$
as well as,
$$
x_\lambda = x_0 - \frac{\lambda}{2}y_\lambda
~;\qquad\text{\% see Equation~\eqref{eq:Sol_x_x0_y}}
$$

\STATE \textbf{Output}: $(x_{\lambda}, g,v)$
\end{algorithmic}
\end{algorithm}

\textbf{Translating Dual Solution into Primal Solution:} Next we need to show how to translate the dual solution of Eq.~\eqref{eq:DualGuaranteesNormBall} into a primal solution for Problem~\eqref{eq:problem_NormBallAppendix}. This can be done as follows:
given $\barlambda$ we compute a primal solution as follows, 
\begin{align}
\label{eq:LagrangianXbar}
\bar{x}: = \min_{x\in\reals^n}  \mathcal L(x,\barlambda)
\end{align}
Note that $\mathcal L(x,\barlambda)$ is strongly-convex in $x$, and therefore the solution is unique.
Next we will  show the following closed form expression for $\bar{x}$,
\begin{align}\label{eq:Sol_x_x0_y}
\bar{x} = x_0 - \frac{\barlambda}{2}\Pi_*\left( \frac{2x_0}{\barlambda}\right)
\end{align}
To see this, note that similarly to what we do in Eq.~\eqref{eq:d_x_y_dual}, we can write $\bar{x}$ (from Eq.~\eqref{eq:LagrangianXbar}) as the solution of the following problem,
$$
\min_{x\in\reals^n} \max_{y: P_*(y)\leq 1}M(x,y;\barlambda):=  \min_{x\in\reals^n} \max_{y: P_*(y)\leq 1}\left\{\|x_0  - x\|^2  + \barlambda (x^\top y-1) \right\}
$$
This is a minimax problem, with a solution $\bx$. Now, similarly to what we did 
 in Eq.~\eqref{eq:d_x_y_dual} we can show that the maxmin solution,  $\bar{y}: = \argmax_{y: P_*(y)\leq 1}\min_{x\in\reals^n} M(x,y;\barlambda)$ satisfies,
$$
\bar{y}: = \argmin_{y:P_*(y)\leq 1} \left\|y-\frac{2x_0}{\barlambda}\right\|^2 = \Pi_*\left( \frac{2x_0}{\barlambda} \right)
$$
Next we will show  that in our case one can extract $\bx$ from $\by$ as follows,
$$
\bar{x}: = \argmin_{x\in\reals^n}M(x,\bar{y};\barlambda) = x_0 -\frac{\barlambda}{2}\bar{y} = x_0 - \frac{\barlambda}{2}\Pi_*\left( \frac{2x_0}{\barlambda}\right)
$$
where we have used the fact that $M(\cdot,\by)$ is quadratic in $x$ and therefore admits a closed from solution $x_0 -\frac{\barlambda}{2}\bar{y}$. This establishes Eq.~\eqref{eq:Sol_x_x0_y}.

The above follows due to the next lemma (see proof in Sec.~\ref{sec:Prooflemma:MinimaxMaximinSol}),
\begin{lemma}\label{lemma:MinimaxMaximinSol}
Let $X,Y$ be convex sets, and $M: X\times Y\mapsto \reals$, be strongly-convex in $x$, and concave in $y$.
Also, assume that $\bx$ and $\by$ are respective solutions of the minimax and maximin problems, i.e.,
$$\bx = \argmin_{x\in X}\max_{y\in Y}M(x,y)~;~~~
\&~~~ \by = \argmax_{y\in Y}\min_{x\in X}M(x,y)~,
$$
 and assume that $\bar{x},\bar{y}$ are unique optimal solutions.
Then given  $\bar{y}$ one can compute $\bar{x}$ as follows,
$$
\bar{x} = \argmin_{x\in X}M(x,\bar{y})~.
$$
\end{lemma}
Note that for our problem $\bx, \bar{y}$   are indeed unique, since both can be described as optimal solutions to (different) strongly-convex minimization problems.

\textbf{Translating Dual Guarantees into Primal Guarantees:}
Let $\eps>0$, and define $\teps = O(\eps^4)$ (similarly to how we define in  Theorem~\ref{theorem:Main}).
Now, given $\barlambda$ that satisfies $d(\lambda^*) -d(\barlambda)\leq \teps$, and $\bar{x}$ of Eq.~\eqref{eq:Sol_x_x0_y} that satisfies
$
\bar{x} = \argmin_{x\in\reals^n}  \mathcal L(x,\barlambda)
$, we would like to show that this $\bar{x}$ is an $\eps$-optimal solution to Problem~\eqref{eq:problem_NormBallAppendix}. This can be done along the exact same lines as we do in our  proof of Theorem~\ref{theorem:Main}, and we therefore omit the details. 

The only difference now, is that we compute exact rather than approximate oracles, and that the dual problem is now one dimensional i.e., $m=1$. We also use the concave bisection algorithm (Alg.~\ref{alg:1DimConc}) as our cutting plane method.

Finally, note that in our case $H: = \max_{x\in\K}|P(x)-1| =1$.

\textbf{Remark:}
It is important to note that Lemma~\ref{lemma:GradientLemmaHighDim} holds as it is (with $m=1$) for unit norm ball constraints, even in the case of non-smooth norms (note that the smoothness parameter does not play a role in this lemma).  This is the important ingredient in applying the exact argumentation in the proof of Theorem~\ref{theorem:Main}, in order to translate the dual guarantees into primal guarantees.

\textbf{Extending Duality Beyond  Norms:}
Given a compact convex set $\K\subseteq\reals^n$ that contains the origin, we can define its polar set as follows,
$$
\K^*: =  \{z\in\reals^n: \max_{x\in\K}z^\top x\leq 1 \}~.
$$
Thus, very similarly to what we have done for norms, one can show that given a projection oracle $\Pi_{\K^*}$ onto $\K^*$, one can use it to approximately project onto $\K$. And the oracle complexity is  logarithmic in the (inverse) target accuracy.

\subsection{Proof of Lemma~\ref{lemma:MinimaxMaximinSol}}
\label{sec:Prooflemma:MinimaxMaximinSol}
\begin{proof}
First notice that due to strong duality the following holds,
$$
\max_{y\in Y}M(\bx,y) = \min_{x\in X}M(x,\by)
$$
Using the above we may write the following,
\begin{align*}
0  \leq M(\bx,\by) - \min_{x\in X}M(x,\by) 
 =
 M(\bx,\by) - \max_{y\in Y}M(\bx,y) \leq 0
\end{align*}
Which implies that,
$$
 \min_{x\in X}M(x,\by) =  M(\bx,\by) 
$$
Since $M(\cdot,\bar{y})$ is strongly-convex in $x$ then its minimizer is unique and therefore,
$$
\bx: = \argmin_{x\in X}M(x,\by)~.
$$
which concludes the proof.
\end{proof}

\newpage
\section{Proofs for Section~\ref{sec:HighDimcase},  the Case of Multiple Constraint,  }
\subsection{Proof of Lemma~\ref{lemma:GradientLemmaHighDim}}\label{Proof_lemma:GradientLemmaHighDim}
\begin{proof}
\textbf{First part, proof of statements \textbf{(i)}:} 

First note that since each $h_i$ is $G$-Lipschitz we have, 
$$
\|\bh(x) - \bh(y)\|\leq \sqrt{m}G\|x-y\|~.
$$

Now, let us first show that the minimizers $x^*_\lambda$ are continuous in $\lambda$.
Indeed let $\lambda_1,\lambda_2\geq 0$, by $2$-strong-convexity of $L(\cdot, \lambda_1)$,
\begin{align*}
\| x^*_{\lambda_2} - x^*_{\lambda_1}\|^2 
&\leq
L(x^*_{\lambda_2},\lambda_1) - L(x^*_{\lambda_1},\lambda_1)  \\
& =
L(x^*_{\lambda_2},\lambda_2) - L(x^*_{\lambda_1},\lambda_2)  + (\lambda_1-\lambda_2)^\top ( \bh(x^*_{\lambda_2})-\bh(x^*_{\lambda_1}))\\
&\leq
0 + \sqrt{m}G\|\lambda_1-\lambda_2\|\cdot\| x^*_{\lambda_2}-x^*_{\lambda_1}\|~,
\end{align*}
where in the first line we use the strong-convexity of $L(\cdot, \lambda_1)$, and the last line uses the optimality of $x^*_{\lambda_2}$, and Lipschitzness of $h(\cdot)$. The above immediately implies that,
\begin{align}\label{eq:LipschitznessGradHD}
\| x^*_{\lambda_2}-x^*_{\lambda_1}\| \leq \sqrt{m}G\|\lambda_1-\lambda_2\|
\end{align}
\textbf{ Proof of existence:}
We are now ready to prove the existence of gradients, as well as $\nabla d(\lambda)=  \bh(x^*_\lambda)$.
Indeed let  a scalar $\eps>0$, and two vectors $\lambda,v \in \reals^n$ such $\lambda,\lambda+\eps v\geq 0$ (elementwise).
In this case,
\begin{align*} 
d(\lambda+\eps v) - d(\lambda) =
L(x^*_{\lambda+\eps v},\lambda+\eps v) - L(x^*_{\lambda},\lambda)  
 \leq
L(x^*_{\lambda},\lambda+\eps v) - L(x^*_{\lambda},\lambda)  
=\eps v^\top \bh(x^*_\lambda)
\end{align*}
Similarly, we can show,
\begin{align*}
d(\lambda+\eps v) - d(\lambda) =
L(x^*_{\lambda+\eps v},\lambda+\eps v) - L(x^*_{\lambda},\lambda)  
 \geq 
 \eps v^\top \bh(x^*_{\lambda+\eps v})
\end{align*}
Thus,
\begin{align*}
 v^\top\bh(x^*_{\lambda+\delta})
\leq \frac{d(\lambda+\eps v) - d(\lambda) }{\eps} \leq  v^\top\bh(x^*_\lambda)
\end{align*}
Taking $\eps\to 0$, and using Eq.~\eqref{eq:LipschitznessGradHD} together with the Lipschitz continuity of $\bh(\cdot)$, implies
$\nabla d(\lambda) : =  \bh(x^*_\lambda)$.

\textbf{Proof of smoothness:} using  Eq.~\eqref{eq:LipschitznessGradHD} implies,
\begin{align}\label{eq:SmoothnessHD}
\|\nabla d(\lambda_1)-\nabla d(\lambda_2)\| = \|\bh(x^*_{\lambda_1}) - \bh(x^*_{\lambda_2})\| \leq \sqrt{m}G\| x^*_{\lambda_1}-x^*_{\lambda_2}\| \leq  mG^2\|\lambda_1-\lambda_2\|~.
\end{align}

\textbf{Second part, proof of statements \textbf{(ii)}:}  Here we show that $\forall \lambda\geq 0$ we have
$$
d(\lambda) - d(\lambda_*)   \leq m^2G^2\|\lambda -\lambda_*\|_\infty^2 + mH\|\lambda-\lambda_*\|_\infty~.
$$
Indeed Eq.~\eqref{eq:SmoothnessHD} shows that $d(\cdot)$ is $2mG^2$-smooth and thus $\forall \lambda\geq0$,
$$
d(\lambda) - d(\lambda_*) \geq \nabla d(\lambda_*)^\top (\lambda-\lambda_*) - mG^2 \|\lambda - \lambda_*\|^2~.
$$
Re-arranging and using $\nabla d(\lambda_*) = \bh(x^*)$ gives,
$$
d(\lambda_*) - d(\lambda) \leq mG^2\|\lambda -\lambda_*\|^2 + \|\bh(x^*)\|\cdot\|\lambda-\lambda_*\|
$$
Since we assume that $\max_{x: \forall j,~h_j(x)\leq 0}|h_i(x)| \leq H,~\forall i\in[n]$, this means that $\| \bh(x^*)\| \leq \sqrt{m}H$. Using this together with $\|y\|_2\leq \sqrt{m}\|y\|_\infty ,\forall y\in\reals^m$ establishes this part.

\textbf{Last part, proof of statements \textbf{(iii)}:} 
	To see that $x^*_{\lambda_*} = x^*$, recall that $L(x,\lambda)$ is convex-concave and therefore strong-duality applies, meaning,
	\begin{align*}
	& L(x^*_{\lambda_*},\lambda_*) := \max_{\lambda\geq 0}\min_{x\in\reals^n} L(x,\lambda)  
	=  \min_{x\in\reals^n}\max_{\lambda\geq 0} L(x,\lambda) = \|x^*-x_0\|^2~.
	\end{align*}
	Using the above together with the definition of $x^*_{\lambda_*}$ and the $2$-strong-convexity of $L(\cdot,\lambda_*)$ implies,
	\begin{align*}
	&\|x^* - x^*_{\lambda_*}\|^2 \leq L(x^*,\lambda_*) - L(x^*_{\lambda_*},\lambda_*)  \\
	& = \left( \|x^*-x_0\|^2 +(\lambda_*)^\top \bh(x^*)\right) - \|x^*-x_0\|^2 \leq 0
	\end{align*}
	where we used $(\lambda_*)^\top \bh(x^*) \leq 0$, which holds since $\lambda_*\geq 0$, and $x^*$ is a feasible solution. Thus $x^*_{\lambda_*} = x^*$.

\end{proof}

\subsection{Proof of Corollary~\ref{corollary:SetSolutions}}\label{proof:corollary1}
\begin{proof}
Recall that in the dual problem~\eqref{eq:dual}, then $\D: = \{\lambda\in\reals^m:~\forall i\in[m];~ \lambda^i\in [  0,R]  \}$.
Denote  the $i^{\rm{th}}$ component of the optimal solution by $\lambda_*^{(i)}$.
And for each $i$ define a segment $S_i\subset \reals_{+}$ as follows:

If $\lambda_*^{(i)}\leq R$, take $S_i = [\lambda_*^{(i)},\lambda_*^{(i)}+r(\eps)]$.
Otherwise, take $S_i = [\lambda_*^{(i)}-r(\eps),\lambda_*^{(i)}]$.

Clearly, the $m^{\rm{th}}$ dimensional box, $\B: = \times_{i=1}^m S_i $, is of radius $r(\eps)$. Using Lemma~\ref{lemma:GradientLemmaHighDim} (part \textbf{(ii)}) , it immediately follows that
$\B$ is contained in the set of $\eps$-optimal solutions in $\D$.

\textbf{Remark:} note that in the proof we assume that $ \frac{1}{2m}\min\{1/H,1/G\} \leq R$, which leads to $\B$ being contained in $\D$. If this is not the case we can always increase $R$.
\end{proof}

\subsection{Proof of Lemma~\ref{lemma:GradOracleLemmaHD}}\label{proof:lemma4.3}
\begin{proof} \label{proof_lemma:GradOracleLemmaHD}
Recall that Alg.~\ref{alg:MinimizingPrimalHighDim} outputs $x\in\reals^n$ such that,
\begin{align}\label{eq:Alg2GuaranteeHD}
L(x,\lambda) - L(x^*_\lambda,\lambda) \leq \teps
\end{align}
Thus we immediately get,
$$
0\leq v - d(\lambda): = L(x,\lambda) - L(x^*_\lambda,\lambda) \leq\teps~.
$$
Using Eq.~\eqref{eq:Alg2GuaranteeHD} together with the $2$-strong-convexity of $L(\cdot,\lambda)$, and with the optimality of $x^*_\lambda$ implies,
$$
\|x - x^*_\lambda\|^2 \leq L(x,\lambda) - L(x^*_\lambda,\lambda) \leq \teps~.
$$
Finally, combining the above with  the Lipschitz continuity  of $\bh(\cdot)$ gives,
$$
\|g-\nabla d(\lambda)\| = \|\bh(x) - \bh(x^*_\lambda)\| \leq \sqrt{m}G \|x-x^*_\lambda \| \leq \sqrt{mG^2\teps}~.
$$
\textbf{Runtime:} 
The guarantees above are independent of the method used for (approximately) minimizing the objective $\min_x L(x,\lambda)$.
It is immediate to see that $L(\cdot,\lambda)$ is $2$-strongly-convex and $2+L\|\lambda\|_1$ smooth objective (recall each $h_i$ is $L$-smooth).
Thus, using Nesterov's method with $\alpha=2, \beta = 2+ L\|\lambda\|_1$, it finds an $\teps$-optimal solution within
$T_{\rm{Internal}}:=O\left(\sqrt{1+ 0.5 L\|\lambda\|_1} \log\left(\frac{(1+ 0.5L\|\lambda\|_1)\|x_0-x^*_\lambda\|^2}{\teps}\right) \right)$ iterations.
Thus the total runtime of Alg.~\ref{alg:MinimizingPrimal} in this case is $O(nT_{\rm{Internal}})$.

To simplify the bound, notice that $\|\lambda\|_1\leq m\|\lambda\|_\infty \leq mR$, and   that $x^*_\lambda$ is $\sqrt{m}G$-Lipschitz continuous in $\lambda$ (see Lemma~\ref{lemma:GradientLemmaHighDim}). Thus letting $x^*$ be the optimal solution to the projection problem (Eq.~\eqref{problem}), since we assume $\|x^*-x_0\|\leq B$, then for any $\lambda \in\D$,
\begin{align*}
\|x_0 - x^*_\lambda\|^2 
& \leq 2\left(\|x_0 - x^*\|^2+\|x^* - x^*_\lambda\|^2   \right) \\
&\leq
2\left(B^2+\|x^*_{\lambda*} - x^*_\lambda\|^2   \right) \\
&\leq
2(B^2 + m^2G^2R^2)~,
\end{align*}
where $\lambda_*$ is the optimal dual solution, and we used $x^* = x^*_{\lambda_*}$.
Thus we may bound, 
\begin{align}\label{T_internal}
T_{\rm{Internal}}:=O\left(\sqrt{1+mR L} \log\left(\frac{(1+mR L)(B^2 + m^2G^2R^2)}{\teps}\right) \right).
\end{align}
\end{proof}

\subsection{Proof of Lemma~\ref{lemma:GradIneqSmooth}}\label{proof:lemma4.4}
\begin{proof}
Let us consider the following function, 
$$
C(\lambda): = F(\lambda) - \nabla F(\lambda_*)^\top (\lambda-\lambda_*)
$$
Clearly $C(\cdot)$ is also concave, and it global maximum is obtained in $\lambda_*$ since $\nabla C(\lambda_*) = 0$.
Also $C(\cdot)$ is $L$-smooth, and therefore $\forall \lambda\in\D,u\in\reals^m$ we have,
$$
C(\lambda+u) \geq C(\lambda) +\nabla C(\lambda)^\top u-\frac{L}{2}\|u\|^2 ~.
$$
Taking  $u=\frac{1}{L}\nabla C(\lambda)$ we get,
$$C(\lambda+u) \geq C(\lambda) +\frac{1}{L}\|\nabla C(\lambda)\|^2-\frac{1}{2L}\|\nabla C(\lambda)\|^2~.$$
Thus $\forall \lambda\in\D$,
\begin{align*}
\|\nabla C(\lambda)\|^2 &\leq 2L \big( C(\lambda+u) - C(\lambda) \big)\\
&  \leq 2L \big( C(\lambda_*) - C(\lambda) \big)~,
\end{align*}
where we have used the fact that $\lambda_*$ is the global maximum of $C(\cdot)$.
Now using the above together with $C(\lambda): = F(\lambda) - \nabla F(\lambda_*)^\top (\lambda-\lambda_*)$, we obtain,
\begin{align*}
\|\nabla F(\lambda) - \nabla F(\lambda_*)\|^2 
&  \leq 2L \big( F(\lambda_*) - F(\lambda) \big) + 2L \nabla F(\lambda_*)^\top (\lambda-\lambda_*) \\
&\leq 2L \big( F(\lambda_*) - F(\lambda) \big) ~,
\end{align*}
where the last inequality uses the fact that $\lambda_*$ is the maximum of $F(\cdot)$ over $\D$, and thus $\forall \lambda\in\D;~
\nabla F(\lambda_*)^\top (\lambda-\lambda_*)\leq 0$.
\end{proof}

%
%

\ignore{
\newpage
\section{Old Stuff}
\newpage
\subsection{Cutting Plane Method with Approximate Oracles}
Here we discuss the case where we would like to minimize a function $g:\K \mapsto \reals$, and we have $\eps$-approximate access to gradient and value oracles for $g$. We rely on the lecture notes  of Anatoli Jouditski (Chapter 3 Section 3.1).
Our goal is to find an $\eps$- approximate solution to $d(\lambda_*): = \min_{\lambda\in \K} d(\lambda)$, where $\lambda_*\in \arg\max_{\lambda\in\K} d(\lambda)$.

Formally assume that we have a closed convex set $\K$ with diameter $D = \sqrt{m}\bar\lambda$. And assume that at a given query point $\lambda\in\K$ we may access oracles $\tilde g$ and $\G$ that provide $\eps$-accurate estimates of $d(\lambda), \nabla d(\lambda)$.

\begin{algorithm}[H]
	\caption{Cutting plane scheme with approximate oracle}
	\label{alg:1.5CuttingPlane}
	\begin{algorithmic}
		\STATE \textbf{Input}:  $\lambda_0$  is a starting point, 	$\lambda_{\min} =0$,  $\lambda_{\max} = \bar{\lambda}$,  the box $\K_0 = \{0\leq \lambda^i \leq \bar \lambda, i =  1, \ldots, m\}$ is the initial localiser, $\e$ is the required accuracy
		\FOR{$t=1 \ldots T$}
		\STATE Choose $\lambda_t$ to be a center of gravity of the current localiser  $\K_{t-1}$ 
		\IF{
			$\|\nabla d(\lambda_t)\|\leq \e$
		} 
		\STATE{Terminate} 
		\ELSE 
		\STATE { Define the new localiser:}\\
		$\K_{t} = \{\lambda \in \K_{t-1}|  (\lambda -\lambda_t)^T\nabla d(\lambda_t)\geq 0\}$
		\STATE   {$t \gets t+1$}  
		\ENDIF
		\ENDFOR
		\STATE \textbf{return} $\lambda_* \in \{\lambda_t \}_{t\leq T} $ such that $d(\lambda_*) = \max_{t<T} d(\lambda_t)$
	\end{algorithmic}
\end{algorithm}

\begin{enumerate}
	\item Now upon choosing a point $x_t$, the oracles provides us with $\tilde d(\lambda_t)\in\reals$, and $\G(\lambda_t)\in\reals^n$ such that 
	$|\tilde d(\lambda_t)-d(\lambda_t)|\leq \eps$, and $\|\G(\lambda_t) -\nabla d(\lambda_t)\| > \eps$.
	By concavity, for any point $\lambda\in\K$ such that $\G(\lambda_t) \cdot(\lambda - \lambda_t)\leq 0$ we have,
	\begin{align*}
	d(\lambda) &\leq d(\lambda_t) + \nabla d(\lambda_t)\cdot(\lambda-\lambda_t) \\
	&\leq d(\lambda_t) + \G(\lambda_t)\cdot(\lambda-\lambda_t) + \eps D \\
	&\leq d(\lambda_t)  + \eps D 
	\end{align*}
	Thus, we can divide into two cases. \\
	\textbf{Case (i)} If there exists an optimal point $\lambda_*\in \K$ such that $\G(\lambda_t) \cdot(\lambda-\lambda_t)\geq 0$, then by the above we have
	$$
	d(\lambda_t)\geq d(\lambda_*) - \eps D
	$$
	And by the assumption of the oracle we also have,
	$$
	\tilde d(\lambda_t) \geq d(\lambda_t)-\eps \geq d(\lambda_*) - \eps (D+1)
	$$
	
	\textbf{Case (ii)} If all optimizers, $x^*$ satisfy $\G(\lambda_t) \cdot(\lambda-\lambda_t)\leq 0$, and thus they belong to,
	$$
	\K_t: = \{\lambda\in\K : \G(\lambda_t) \cdot(\lambda-\lambda_t)\geq 0 \}~.
	$$
	Thus, we can iterate the process. We start with a localizer $\K_0 = \K$, choose $\lambda_t\in \text{Interior}(\K_{t-1})$, and check whether $\|d(\lambda_t)\|\leq \eps$. If so, we found and $\eps D$-optimal solution. Otherwise, we can iterate and choose,
	$$
	\K_t: = \{\lambda\in\K : \G(\lambda_t) \cdot(\lambda-\lambda_t)\leq 0 \}~.
	$$
	The approximate solution after $T$-steps is $\bar{\lambda}_T\in\{\lambda_1,\ldots,\lambda_T\}$ such
	$$
	g(\bar{\lambda}_T): = \max_{t\leq T} \tilde d(\lambda_t)~. 
	$$
	(We can similarly prove Lemma 3.2.1 from Jouditski's lecture notes.)
\end{enumerate}

\subsection{Ellipsoid method (Nemirovsky \& Yudin 1979).}
We suggest to apply the Ellipsoid method due to Nemirovsky \& Yudin (1979) in this case to the concave maximization problem (\cref{dual}). 
The only difference is that the original method makes use of the exact gradient of the objective $\nabla d(\lambda)$. In contrast, in our case we only can use the $\bar\e$-accurate oracle information, provided by the Algorithm 2. We show that by simply replacing $\nabla d(\lambda_t)$ by the oracle information $\G_t$ we obtain $\e$ accurate solution  with the same convergence rate.

We write the case without fucntional constraints, since the constraints in this case are the simple box constraints: $\K = \left\{\lambda\in\R^m: 0 \leq  \lambda^i \leq \bar \lambda \forall i = 1\ldots,m \right\}$.
\begin{algorithm}[h]
	\caption{$m$-dim Concave Optimization in dual space}
	\label{alg:2Ellipsoid}
	\begin{algorithmic}
		\STATE \textbf{Input}:  $\bar{\lambda}$, $\eps$
		\STATE {Initialization}: 
		The box $\K = \{\lambda\in\R^m| \lambda_{\min} \leq \lambda^i \leq  \lambda_{\max} \}$, $\tilde{\eps} =\frac{\eps^2}{2G^2}$\\
		$\alpha(m) = \left\{\frac{m^2}{m^2 - 1}\right\}^{1/4}$, 
		$\gamma(m) = \sqrt{\frac{m-1}{m-2}}\alpha(m) $, $ \kappa^m(m) = \frac{n^2}{n^2 - 1}\sqrt{\frac{m-1}{m+1}} $\\
		Choose $m\times  m$ nonsingular matrix $B_0$ and $\lambda_0$ is center of the box,  such that $\K_0   =  \{B_0 u + \lambda_0|u^Tu \leq 1\}$ contains $\K.$ Choose $\beta >0$ such that $\beta \leq  \frac{EllOut(\K)}{EllOut(\K_0)}.$
		 \FOR{$t=1 \ldots T$}
		 \IF{
		 	$\lambda_t  \notin int \K$
	 	} 
		\STATE{1)} Call step $t$ non-productive, find nonzero $e_t$ such that $ (\lambda -\lambda_t)^Te_t\leq 0 \forall  \lambda  \in \K$
		\ELSE 
		\STATE {2)  Call step $t$ productive. Call the oracle to compute the quantities:} $(d(\lambda_t), \nabla d(\lambda_t))  =  (\|x_t-x_0\|^2 + \lambda_t h(x^*_{\lambda_t}), h(x^*_{\lambda_t})) \gets \G(\lambda_t,\tilde{\eps})$.  \\Set $e_t = -\G(\lambda_t,\tilde \e).$
		\ENDIF 
		\STATE{3) Set } 
	 $p = \frac{B^T_{t-1}e_t}{\sqrt{e_t^TB_{t-1}B_{t-1^Te_t}}},$\\	
	 $B_t = \alpha(m) B_{t-1} - \gamma(m)(B_{t-1}p)p^T,$\\  $\lambda_{t+1} = \lambda_t -  \frac{1}{m+1}B_{t-1}p$\\
	 A new localiser 
	 $\K_t = \left\{\lambda = B_t u + \lambda_t|u^Tu \leq 1\right\}$
		\IF{$\kappa^t(m) \leq \eps\beta$}
		\STATE \textbf{return} $\lambda_* = \arg\max_{t<T ~ productive~steps} d(\lambda_t)$
		\ELSE 
		\STATE   {$t \gets t+1$}  
		\ENDIF
		\ENDFOR
	\end{algorithmic}
\end{algorithm}
\begin{theorem}[Yuditsky 3.5.1]
	For a given accuracy $\e$ Ellipsoid method $Ell(\e)$, as applied to a problem (\ref{dual}) of dimentionality $m$, terminates in no more than 
	$A(Ell(\e)) = [\frac{\ln (1/\beta\e)}{\ln (1/\kappa(m))}] \leq [2 m(m-1)\ln (1/\beta\e)]$ and solves P with relative accuracy $\e$. Given the direction defining $t$-th cut, it takes $O(m^2)$ arithmetical iterations to update $B_{t-1},\lambda_t$ to $B_t, \lambda_{t+1}$. \end{theorem}
\begin{proof}
	From Lemma 3.5.1 it follows that $$EllOut(\K_t) \leq \kappa^t(m)EllOut(\K_0)\leq \beta^{-1}\kappa^t(m)EllOut(\K).$$ Due to the termination criterion, definition of $\e$, $$EllOut(\K_t) \leq \e EllOut(\K).$$ \iu{Then, refer to Proposition 3.3.1 of the cutting plane scheme, the authors say it can be literally repeated for the ellipsoid method. We should adopt it carefully to the oracle case, similarly to the proof of Lemma 3.2.1 for the oracle information, which is shown below in the section 5.2.} \end{proof}

	\subsection{Ellipsoid method with Approximate oracles proof continuation}
	The previous discussion cannot be applied directly since there exist non-productive steps, not all the cut-offs are like in cutting plane, sometimes we can appear outside the original set.
		
	Actually, the flow should be as follows. If the algorithm terminated - everything is straightforward and the accuracy is achieved. If not, then we have 
	$$EllOut(\K_T) \leq \e EllOut(\K).$$ Let $x^*$ be an optimal solution to the problem and let $\K^{\e} = \lambda_* + \e(\K - \lambda_* ).$ Then $\K^{\e} $ is the closed and bounded subset of $\K.$ Also, we have $EllOut(\K^{\e}) = \e EllOut(\K) > EllOut(\K_T).$  Hence, $\K^{\e}$ cannot be a subset of $\K_T,$ i.e., $\E  y\in \K^{\e}\setminus \K_T.$ Now, $y\in \K$, and since it is not in $\K_T$, it was catted off at some step $t\leq T.$ That  means: 
	\begin{align}\label{cut}
	e_t^T(y-\lambda_t)>0
	\end{align} Also, since $y\in\K^{\e}$, we have representation \begin{align}\label{hole}
	y = (1-\e)\lambda_* +  \e z
	\end{align} with certain $z\in\K.$
	
Next, we prove that in fact $t$-th step is productive. 
	\iu{(Here we need carefully to adapt it to separating hyper-planes instead of constraint gradients)}
 Then, since $t$ is a productive step, we have $e_t = \nabla d(\lambda_t).$ Therefore, by (\ref{cut} ) we have $g(y) \leq d(\lambda_t).$ Thus, combining it with (\ref{hole}) we get
 $$ d(\lambda_t) - g^* < g(y) - g^*  \leq \e(\max_{\K} g - g^*).$$

}
\end{document}